\documentclass{amsart}
\usepackage{amssymb,latexsym,amscd,array}
\usepackage[leqno]{amsmath}
\usepackage{enumerate}
\usepackage{exscale}
\usepackage[centertags]{amsmath}
\usepackage{amsthm}
\usepackage{dsfont}
\usepackage[all]{xy}
\usepackage{lscape}
\usepackage{pdflscape}
\usepackage{color}
\usepackage{graphicx}
\usepackage{url}
\usepackage[neveradjust]{paralist}

\newtheorem*{prb}{Problem}
\newtheorem{thm}{Theorem}[section]
\newtheorem{prop}[thm]{Proposition}
\newtheorem{lem}[thm]{Lemma}
\newtheorem{cor}[thm]{Corollary}
   \theoremstyle{definition}

\newtheorem{defn}[thm]{Definition}
   \theoremstyle{remark}
\newtheorem{ex}[thm]{Example}
\newtheorem{rem}[thm]{Remark}

\newtheorem{ntn}[thm]{Notation}

\newcommand{\mba}{\mathbb{A}}

\newcommand{\mbc}{{\mathbb{C}}}
\newcommand{\mbd}{\mathbb{D}}

\newcommand{\mbg}{\mathbb{G}}

\newcommand{\mbl}{\mathbb{L}}

\newcommand{\mbn}{\mathbb{N}}

\newcommand{\mbp}{\mathbb{P}}
\newcommand{\mbq}{\mathbb{Q}}
\newcommand{\mbr}{\mathbb{R}}

\newcommand{\mbz}{\mathbb{Z}}

\newcommand{\ra}{\rightarrow}
\newcommand{\lra}{\longrightarrow}

\newcommand{\mcd}{\mathcal{D}}

\newcommand{\mcg}{\mathcal{G}}
\newcommand{\mch}{\mathcal{H}}
\newcommand{\mci}{\mathcal{I}}

\newcommand{\mck}{\mathcal{K}}
\newcommand{\mcl}{\mathcal{L}}
\newcommand{\mcm}{\mathcal{M}}
\newcommand{\mcn}{\mathcal{N}}
\newcommand{\mco}{\mathcal{O}}

\newcommand{\mcs}{\mathcal{S}}

\newcommand{\mcv}{\mathcal{V}}

\newcommand{\hatV}{{\widehat V}}

\newcommand{\p}{\partial}

\def\schluss{\hfill\ensuremath{\diamond}}

\newcommand{\hatmcm}{{\widehat\mcm}}

\newcommand{\boldone}{{\mathbf{1}}}
\newcommand{\bolda}{{\mathbf{a}}}
\newcommand{\boldb}{{\mathbf{b}}}

\newcommand{\boldk}{{\boldsymbol{k}}}

\newcommand{\boldu}{{\mathbf{u}}}
\newcommand{\boldv}{{\mathbf{v}}}

\newcommand{\del}{\partial}

\newcommand{\dual}{*}

\newcommand{\image}{\operatorname{im}}
\newcommand{\into}{\hookrightarrow}
\newcommand{\onto}{\twoheadrightarrow}
\renewcommand{\to}{\longrightarrow}

\newcommand{\rmM}{{\mathrm{M}}}
\newcommand{\rmR}{{\mathrm{R}}}

\newcommand{\minus}{\smallsetminus}

\newcommand{\calD}{\mathcal{D}}

\newcommand{\calH}{\mathcal{H}}

\newcommand{\calL}{\mathcal{L}}
\newcommand{\calM}{\mathcal{M}}

\newcommand{\calO}{\mathcal{O}}

\newcommand{\rmD}{\mathrm{D}}

\newcommand{\frakm}{{\mathfrak{m}}}

\newcommand{\frakt}{{\mathfrak{t}}}
\newcommand{\frakx}{{\mathfrak{x}}}
\newcommand{\fraky}{{\mathfrak{y}}}
\newcommand{\frakz}{{\mathfrak{z}}}

\renewcommand{\AA}{\mathbb{A}}
\newcommand{\CC}{\mathbb{C}}
\newcommand{\DD}{\mathbb{D}}

\newcommand{\GG}{\mathbb{G}}

\newcommand{\NN}{\mathbb{N}}

\newcommand{\QQ}{\mathbb{Q}}
\newcommand{\RR}{\mathbb{R}}
\newcommand{\ZZ}{\mathbb{Z}}

\newcommand{\partitions}{\dashv}

\newcommand{\TVmap}{h}
\newcommand{\TXmap}{\varphi}
\newcommand{\XVmap}{i}

\DeclareMathOperator{\cone}{\textup{cone}}
\DeclareMathOperator{\Cone}{\textup{Cone}}

\DeclareMathOperator{\Dmod}{\textup{Dmod}}

\DeclareMathOperator{\FL}{\textup{FL}}

\DeclareMathOperator{\gr}{\textup{gr}}
\DeclareMathOperator{\IC}{\textup{IC}}
\DeclareMathOperator{\id}{\textup{id}}
\DeclareMathOperator{\HH}{\textup{H}}
\DeclareMathOperator{\IH}{\textup{IH}}
\DeclareMathOperator{\ih}{\textup{ih}}
\DeclareMathOperator{\Int}{\textup{Int}}

\DeclareMathOperator{\MHM}{\textup{MHM}}

\DeclareMathOperator{\pr}{\textup{pr}}

\DeclareMathOperator{\qdeg}{\textup{qdeg}\,}
\DeclareMathOperator{\sRes}{\textup{sRes}\,}
\DeclareMathOperator{\supp}{\textup{supp}}
\DeclareMathOperator{\Spec}{\textup{Spec}\,}

\newcommand\oldtext[1]{}
\newcommand\comment[1]{}

\newcommand{\mustata}{{Musta{\c{t}}{\v{a}}}}

\numberwithin{equation}{subsection}

\begin{document}
\title{Weight filtrations on GKZ-systems
}
\author{Thomas Reichelt}
\address{
T.~Reichelt\\
Lehrstuhl für Algebraische Geometrie \\
Universit\"at Mannheim\\
B6 26\\
68159 Mannheim\\
Germany}
\email{reichelt@math.uni-mannheim.de}

\author{Uli Walther}
\address{ U.~Walther\\
  Purdue University\\
  Dept.\ of Mathematics\\
  150 N.\ University St.\\
  West Lafayette, IN 47907\\ USA}
\email{walther@math.purdue.edu}

\begin{abstract} 

Given an integer matrix $A\in\ZZ^{d\times n}$, we study the natural
  mixed Hodge module structure in the sense of Saito on the
  Gau\ss--Manin system attached to the monomial map
  $\TVmap\colon(\CC^*)^d\to\CC^n$ induced by $A$. We completely
  determine in the normal case the associated graded object to the
  weight filtration, by
  computing the intersection complexes with respective multiplicities
  that form its constituents. Our results show that
  these data are purely combinatorial, and not arithmetic, in the
  sense that they only depend on the polyhedral structure of the cone
  of $A$, but not on the semigroup itself. In particular, we extend
  results of de Cataldo, Migliorini and \mustata\ to the setting of
  torus embeddings and give a closed form for the failure of the
  Decomposition Theorem in our context.

If $A$ is homogeneous and if $\beta\in\CC^d$ is an integral but not
strongly resonant parameter, we use a monodromic
Fourier--Laplace transform to carry the mixed Hodge module structure
from the Gau\ss--Manin system to the GKZ-system attached to $A$ and
$\beta$. In case $A$ is derived from a normal reflexive Gorenstein
polytope $P$, Batyrev and Stienstra related certain filtrations on the
generic fiber of the GKZ-system to the mixed Hodge structure on the
cohomology of a generic hyperplane section inside the projective toric
variety induced by $P$. Our formul\ae, phrased in terms of intersection
cohomology groups on induced relative toric varieties, provide the
necessary correction terms to globalize their computation. In
particular, we document that on the GKZ-system the weight filtration
will differ from Batyrev's filtration-by-faces whenever $P$ is not a
simplex: the intersection complexes contributing to the weight
filtration measure the failure of $P$ to be a simplex.

Irrespective of homogeneity, we obtain a purely combinatorial formula
for the length of the Gau\ss--Manin system, and thus for
the corresponding GKZ-system.  In dimension up to three, and for
simplicial semigroups, we give explicit generators of the weight
filtration.

\end{abstract}



\thanks{TR was supported by  DFG Emmy-Noether-Fellowship RE 3567/1-1}
\thanks{UW was
supported by NSF grants 1401392-DMS and 2100288-DMS, and by Simons Foundation
Collaboration Grant for Mathematicians \#580839.}

\maketitle

\tableofcontents

\section{Introduction}

\subsection{The Decomposition Theorem for proper maps}

One of the hallmarks of Hodge-theoretic results in algebraic geometry
is the Decomposition Theorem. For smooth projective maps
between smooth projective varieties this asserts among other things
the degeneration of the Leray spectral sequence for $\QQ$-coefficients
on the second page.  Decomposition Theorems are refinements and
generalizations of the Hard Lefschetz Theorem
for projective varieties; the key ingredient is the purity of the
Hodge structure on cohomology. In this article we study and quantify
an important instance of the failure of purity, and of the Decomposition
Theorem.
In order to state our results, we give the briefest of historical surveys,
and we point to the excellent  account
\cite{CataldoMigliorini-Bulletin} for details.

For singular maps and varieties, things can be rescued by replacing
usual cohomology with intersection cohomology, and in both
instances the statement has a local flavor in the sense that one can
restrict to open subsets of the target. The advantage of intersection
cohomology is that it has nice formal properties such as Poincar\'e
duality, Lefschetz theorems, and K\"unneth formula. While it is not a
homotopy invariant, there is a natural transformation $\HH^i\to\IH^i$
that is an isomorphism on smooth spaces and in general induces a
$\HH^\bullet$-module structure on $\IH^\bullet$. This version of the
Decomposition Theorem,
conjectured by S.~Gel'fand and R.~MacPherson, was proved by
A.~Beilinson, J.~Bernstein, P.~Deligne and O.~Gabber.

The construction allows for generalization of intersection cohomology
to coefficients in a local system $L_U$, defined on a locally closed
subset $U\subseteq Z=\bar U$. The intersection complex of such a local
system is a constructible complex that extends $L_U$ as a constructible
complex (or the corresponding connection on $U$ as $D$-module). In fact,
the best form of the Decomposition Theorem in the projective case is
in this language: if $f\colon X\to Y$ is a proper map of complex
algebraic varieties then $Rf_*\IC_X$ splits (non-canonically) as a
direct sum of intersection complexes whose supporting sets are induced
from a stratification of $f$.

A particularly interesting case where the decomposition theorem has been well-studied are semi-small maps (cf. \cite{CataldoMigliorini-Bulletin} for a nice survey). These maps arise often in geometric situations:
\begin{itemize}
\item the Springer resolution $f\colon \widetilde{\mcn}\to \mcn$ of the nilpotent cone $\mcn$  of the Lie algebra to the reductive group $G$; 
\item the Hilbert--Chow map between the Hilbert schemes of points $X= (\mbc^2)^{[n]}$ and the $n$-th symmetric product $Y=(\mbc^2)^n / \mcs_n$.  
\end{itemize}


The most explicit
case is perhaps that of a fibration $f\colon X\to Y$ between toric complete
varieties: $\IH^\bullet$ of a complete toric variety can
be written down in purely combinatorial terms, and
\cite{deCataldoMiglioriniMustata} spells out how to write 
$Rf_*(\IC_X)$ as sum of
intersection complexes in terms of face numbers.



\subsection{Non-proper maps}

The moment one moves away from proper maps, direct images of
intersection complexes need no more split into sums of such. For
example, embedding $\CC^*$ into $\CC=\CC^*\sqcup \{pt\}$ leads to a
push-forward $Rf_* \mco_{\mbc^*}$ that naturally contains
$\calO_{\CC^1}$ but there is a nontrivial
cokernel of the form $\calO_{pt}$. At this point one
requires a ``weight'' filtration on $Rf_*\calO_{\CC^*}$ akin to the
one that forms part of Deligne's construction of mixed
Hodge structures on the cohomology of complex varieties.
In the case $\CC^*\into \CC$, level 1 of the weight filtration on $Rf_*(\calO_{\CC^*})$
is $\calO_{\CC^1}$; level 2 is the entire image.

The appropriate powerful hybrid of intersection complexes and
Deligne's weights was constructed by M.~Saito in his theory of mixed
Hodge modules, 
inspired by
the theory of weights for $\ell$-adic sheaves
\cite{SaitoMHM}.
The weight filtration, together with a ``Hodge filtration'' that can
be seen as avatar of the usual Hodge filtration on cohomology, form
the main ingredients of an object in Saito's category of mixed Hodge
modules. For maps between quasi-projective varieties he introduced a natural
geometric filtration on $Rf_*\IC_X$. For proper maps between algebraic
varieties, the weight filtration on $Rf_* \IC_X$ is pure and in particular there is a
Decomposition Theorem: $Rf_*\IC_X$ splits into intersection complexes
and the splitting occurs in the category of mixed Hodge modules.

%
%
%

In several natural
situations properness is not available, and this necessitates
nontrivial weights.
Saito's theory shows that
in general the associated graded pieces of the weight filtration of any
push-forward of a mixed Hodge module split as sums of intersection
complexes, while  for
maps to a point the construction agrees with Deligne's weights.

One is naturally led to a very hard
question, crucial to Saito's theory, on the behavior of pure Hodge
modules under open embeddings.
In the world of toric varieties,
once one gives up on complete fans, the most fundamental situation is
the inclusion of an embedded torus into its (likely singular) closure:
\begin{prb}[Weight Decomposition for Open Tori]
Let $T=(\CC^*)^d$ and consider the monomial map
\begin{eqnarray}\label{eq:h-map}
  \TVmap\colon T&\to& \CC^n=:\hatV\\
  (\frakt_1,\ldots,\frakt_d)=:\frakt&\mapsto& \frakt^A:=(\frakt^{\bolda_1},\ldots,\frakt^{\bolda_n})\nonumber
\end{eqnarray}
where
\[
A=(\bolda_1,\ldots,\bolda_n)\in(\ZZ^d)^n
\]
is an integer  $d\times n$ matrix.
Determine the weight filtration
$\{W_i\}_i$ on $\TVmap_+(\calO_T)$, and for each associated graded
quotient $W_i/W_{i-1}$ indicate the intersection complexes (support and
coefficients) that appear as direct summands of this module.
\end{prb}

To our knowledge, the only other place where
non-proper maps have been studied in this context is the article
\cite{CautisDoddKamnitzer} on certain open subsets of
products of Grassmannians.

\subsection{Results and techniques}

Throughout, $A$ is an integer $d\times n$ matrix satisfying the three
conditions of Notation \ref{ntn-A} and we consider the induced monomial
action of $T$ on $\hatV=\CC^n$ given by the Hadamard product
\begin{eqnarray}\label{eq:T-action}
\mu\colon T\times \hatV &\to& \hatV,\\
(\frakt,\fraky) &\mapsto& \frakt^A\star\fraky:=(\frakt^{\bolda_1}\fraky_1,\ldots,\frakt^{\bolda_n}\fraky_n).
\end{eqnarray}
With $\sigma=\RR_{\geq 0}A$ and any face $\tau$ of $\sigma$,  define
$\boldone_\tau\in\CC^n$ by $(\boldone_\tau)_j=1$ if
$\bolda_j\in\tau$ and zero otherwise. 
Denote $X_\sigma$ (or just $X$) the closure $\Spec(\CC[\NN A])$ of the $T$-orbit through
$\boldone_\sigma=(1,\ldots,1)\in\CC^n$. There is an orbit decomposition 
\[
X=\bigsqcup_{\tau}O_\tau
\]
where the union is over the faces $\tau$ of $\sigma$ and
$O_\tau=\mu(T,\boldone_\tau)$ is the orbit corresponding to $\tau$.
Denote $\QQ_\tau$ the constant
sheaf on the orbit to $\tau$ and write ${}^p\QQ^H_\tau$ for the
corresponding (simple, pure) Hodge module.
With $X$ as in Section
\ref{sec:weightontorus}, $\TVmap$ factors as
\begin{gather}\label{map-factorization}
  \xymatrix{T \ar[r]^{\displaystyle{\TXmap}}
    \ar@/_1pc/[rr]_-{\displaystyle{\TVmap}}
    & X \ar[r]^-{\displaystyle{\XVmap}} &\mbc^n=\hatV}
\end{gather}
Via Kashiwara equivalence along $\XVmap$, one identifies the mixed
Hodge modules on $X$ with those on $\hatV$ supported on $X$.  The
following is a brief list of the  results that we prove in Section
\ref{sec:weightontorus}.

\begin{asparaenum}
  \item Since $\calO_T$ is a (strongly) torus equivariant
    $\calD_T$-module, the $T$-equivariant map $\TVmap$ will produce
    (strongly) equivariant modules $R^i\TVmap_+(\calO_T)$ (and only
    the $0$-th one
    is nonzero since $h$ is affine). The intersection complexes
    appearing in the decomposition of the weight graded parts of
    $R\TVmap_*({}^p\QQ^H_\sigma)$ are equivariant, supported on 
    orbits. The underlying local systems are constant, of the form
    ${}^p\QQ^H_\tau$.

\item We show that two specific functors are isomorphic on equivariant
  sheaves with contracting torus actions. Using this identification,
  we provide a recursive recipe for the exceptional pullback
  $\calH^k
  i^!_\tau(\TVmap_*{}^p\QQ^H_\sigma/W_i\TVmap_*{}^p\QQ^H_\sigma)$ to
  an arbitrary orbit of the monomial action from \eqref{eq:T-action}.

\item We unravel the recursion for $\calH^0
  i^!_\tau\TVmap_*{}^p\QQ^H_\sigma$, for every $\tau$, to provide an
  explicit expression for the multiplicity $\mu^\sigma_\tau(e)$ of the
  constant local system ${}^p\QQ^H_\tau$ in the $(d+e)$-th graded
  weight part of $\TVmap_*{}^p\QQ^H_\sigma$ in terms of an alternating
  sum whose constituents are indexed by flags in the face lattice of
  $\sigma$, see Proposition \ref{prop-chains}. The terms involve
  intersection cohomology dimensions of the affine toric varieties
  $X_{\tau/\gamma}$ associated to the semigroup of the cone
  \[
  \tau/\gamma:=(\tau+\RR\gamma)/\RR\gamma.
  \]

  \item Using some results on intersection cohomology of toric varieties
    by Stanley, and Braden and MacPherson, we express
    $\mu^\sigma_\tau(e)$ as a single intersection cohomology rank on the
    dual affine toric variety $Y_{\sigma/\tau}$, associated to the
    dual of $\sigma/\tau$, see Theorem \ref{thm:weightonhQ}.
\end{asparaenum}

\medskip

There are some immediate noteworthy consequences. First of all,
$\mu^\sigma_\tau(e)$ is a relative quantity in the sense that
$\mu^\sigma_\tau(e)=\mu^{\sigma/\gamma}_{\tau/\gamma}(e)$ for any face
$\gamma$ inside $\tau$. Secondly, the arithmetic properties of
$\sigma$ are inessential: the only information relevant for
$\mu^\sigma_\tau(e)$ is the combinatorics of the polytope obtained
from $\sigma/\tau$ by slicing it with a transversal hyperplane that
``cuts off the vertex''. This is
because the intersection cohomology numbers of $Y_{\sigma/\tau}$ are
entirely combinatorial.

\subsection{Consequences, applications, open problems}

\subsubsection{Hodge structures on GKZ-systems}
Some interesting consequences of (1)-(4) above come from applying
these results to the Fourier--Laplace transform of $\TVmap_+ \mco_T$, a
well-studied $D$-module all by itself.

We briefly recall the notion of an $A$-hypergeometric system in our
setup. Let $R_A=\CC[\del_1,\ldots,\del_n]$ be the polynomial ring
where $\del_i$ stands for the partial differential operator
$\frac{\del}{\del x_i}$. Then denote
$D_A=R_A\langle x_1,\ldots,x_n\rangle$ the Weyl algebra and pick
$\beta\in \CC^d$. Now consider the
left ideal $H_A(\beta)$ of $D_A$ generated by
\[
I_A:=R_A(\{\del^\boldu-\del^\boldv\mid \boldu,\boldv\in\NN^n,\,
A\cdot\boldu=A\cdot\boldv\})
\]
and all
\[
E_i-\beta_i:=\sum_{j=1}^n a_{i,j}x_j\del_j-\beta_i\qquad \text{with }i=1\ldots,d.
\]
The module
\[
M_A^\beta:=D_A/H_A(\beta)
\]
is the $A$-hypergeometric system induced by $A$ and $\beta$
introduced by Gel'fand, Graev, Kapranov and Zelevinsky in
the 1980's.  We refer to \cite{SST}, the surveys
\cite{Sti, RSSW},  and the current literature for more
information on these modules, but highlight some properties.

The \emph{strongly
resonant quasi-degrees} $\sRes(A)$ of $A$ form an infinite discrete hyperplane
arrangement in $\CC^d$ which was introduced in
\cite{SchulzeWalther-ekdi} and used to sharpen a result of Gel'fand et
al.\ by showing that $\beta\not\in\sRes(A)$ is equivalent to
$M_A^\beta$ being the Fourier--Laplace transform of $\TVmap_+(\calO_T^\beta)$
where $\calO_T^\beta$ is described before Theorem
\ref{thm-SWekdi}. In fact, it was Gel'fand and his collaborators that
first observed a connection between $A$-hypergeometric systems and
intersection complexes in \cite[Prop.~3.2]{GKZ-Euler}.

If the semigroup ring $S_A:=\CC[\NN A]\simeq R_A/I_A$ is normal (or,
equivalently, if the semigroup $\NN A$ is saturated in $\ZZ A$) then
$0$ is not strongly resonant. In particular then, the inverse Fourier--Laplace
transform of $M_A^0$ is the module $\TVmap_+(\calO_T)$ from Section \ref{sec:weightontorus}. Since the Fourier--Laplace transform is an
equivalence of categories, our results on $\TVmap_+(\calO_T)$ solve for
normal $S_A$ the longstanding problem of determining the composition
factors for $M_A^0$.

The Fourier--Laplace transform does not necessarily preserve mixed
Hodge module structures in general. However, if one assumes that $I_A$
defines a projective variety, one can use the monodromic
Fourier--Laplace transform which produces the same output as the
Fourier--Laplace transform on $\TVmap_+(\calO_T)$ and does carry mixed
Hodge module structures. In particular, this equips $M_A^0$ with a
natural mixed Hodge module structure inherited from $\TVmap$ (cf. \cite{Reich2}).

There is a filtration-by-faces on a GKZ system, defined via the face
filtration on the semigroup ring: the $(d+k)$-th level of this
filtration is the submodule of $M_A^0$ generated by all monomials
$\del^\boldu\in S_A$ for which $A\cdot \boldu$ is not contained in a
face of dimension $d-k-1$. This filtration was introduced by Batyrev
in his study of the Hodge structure on the cohomology of a generic
hypersurface in a toric variety constructed from a polytope
\cite{Bat4,Sti}. Adolphson and Sperber, and more recently Fang, also
considered the face filtration in \cite{AS,Fang}. We show that this
filtration is bounded above by the weight filtration, and that it
really differs from it for all GKZ-systems whose semigroup cone is not
the cone over a simplex. On can view the error terms that we find as
the necessary ``glue'' that is required to globalize the result of
Batyrev and Stienstra from the generic fiber to the entire
GKZ-system. On the other hand, looking at $\TVmap_+(\calO_T)$, we show
that the corresponding filtration-by-faces always captures the part of
the weight filtration that has maximal dimensional support.

\subsubsection{Applications}

We outline two possible applications of our results;
one is concerned with mirror symmetry, the other comes from commutative
algebra. 

\emph{Local cohomology at toric varieties}: Let $\hat
R=\CC[x_1,\ldots,x_n]$ and suppose $\hat I$ is an ideal of $\hat R$
such that $\hat R/\hat I$ is the semigroup ring $\CC[\NN A]$ for some
matrix $A$ as above. Let $\hat J$ denote the ideal of the variety
comprised of the smaller torus orbits of the variety of $\hat I$. Then
there is a natural triangle
\[
\to \RR\Gamma_{\hat J}(\hat R)\to\RR\Gamma_{\hat I}(\hat R)\to \hat
M_A^0[-c]\stackrel{+1}{\to}
\]
in the category of mixed Hodge modules where the first morphism is the
canonical one and $\hat M_A^0$ is the inverse Fourier--Laplace
transform of $M_A^0$ (\emph{i.e.}, $h_+(\calO_T)$ up to shift). In the
normal case this sequence degenerates and perhaps an inductive
procedure can be used to determine from our formul\ae\ for $\hat
M_A^0$ the intersection complexes in the weight filtration of
$\HH^\bullet_{\hat I}(\hat R)$. In particular, their vanishing (which
at present is an open problem) might be computable.

\emph{Mirror symmetry}: Let $Y_\Sigma$ be a toric variety induced by
the fan $\Sigma$.  The secondary fan of $Y_\Sigma$ induces a toric
variety $M$ and a family of Laurent polynomials over a Zariski open
subset of $M$. This family is known as the Landau--Ginzburg model of
$\Sigma$ and encodes the Gromov--Witten invariants of
$Y_\Sigma$. It turns out that the information relevant to
Gromov--Witten invariants is contained in the smallest weight part of
the Gau\ss --Manin system, compare \cite{Givental1,Givental2,Iritani,ReiSe2,ReiSe3} . It is conjectured that the parts of higher
weight describe mirror symmetry for toric degenerations such as flag
manifolds \cite{IritaniXiao}.
Our results here give concrete data on the GKZ side which one should
want to match to those toric degenerations.

\subsubsection{Open problems}

When $S_A$ is normal, the holonomic rank of $M_A^\beta$ (the dimension
of the holomorphic solution space in a generic point) equals the
volume of the convex hull of the columns of $A$ together with the
origin. In particular, this is an arithmetic quantity, not just
determined by the combinatorics. In contrast, our results show that
the holonomic length is purely combinatorial in that case; it only
depends on the $\mathbf{cd}$-index (see \cite{BayerKlapper-cdIndex})
of the polytope over which $\sigma$ is the cone. This suggests a new
question that deserves study: what is the rank, and more generally the
characteristic cycle, of the Fourier--Laplace transformed intersection
complex $\IC_{X_\tau}$? 
Our results allow for small $d$ direct
calculation of the rank of $\FL(\IC_{X_\tau})$ 
to any chosen
face. In higher dimension one can write down recursions, but making
them explicit is an open question.

Further, having a saturated composition chain for a
$D$-module informs on the irreducible representations in the monodromy
of the solution sheaf. Studying $\FL( \IC_{X_\tau})$ would be the
first step towards a general understanding of the monodromy of
$M_A^0$.

Finally, one should investigate whether one can place mixed Hodge
module structures also on $M_A^\beta$ for other $\beta$. Obviously,
this is doable in the normal case with $\beta\in\NN A$ since
\cite{SchulzeWalther-ekdi} implies that the corresponding $M_A^\beta$
are isomorphic to $M_A^0$ via contiguity operators. Similarly, dual
ideas reveal that for $\beta$ integral and in the cone roughly
opposite to $\NN A$, $M_A^\beta$ agrees with the Fourier--Laplace
transform of $\TVmap_!(\calO_T^\beta)$ and hence also inherits a mixed
Hodge module structure, dual to the one discussed here. For other
integral $\beta$, \cite{Avi,Avi2} describes $\FL^{-1}(M_A^\beta)$ as a
composition of a direct and exceptional direct image, which can be
used to export a MHM structure. Less clear are non-integral $\beta$:
the use of complex Hodge modules allows to equip $M_A^\beta$ with
$\beta\in\RR^d$ with a MHM structure, see Sabbah's MHM project
\cite{Sabbah-webpage}. For certain $\beta$ the Hodge filtration on
$M_A^\beta$ is explicitly computed in \cite{ReiSe3}.

\subsection{Acknowledgments}

We would like to thank Qianyu Chen and Bradley Dirks for catching an
indexing error, and them as well as Andra\'s L\H{o}rincz, Takuro
Mochizuki, Mircea \mustata, Claude Sabbah and Duco van Straten for
helpful and engaging conversations about this work. We are indebted to
Avi Steiner for a very careful reading which resulted in many
improvements, both technical and editorial. We wish to thank the
referee for useful suggestions and corrections.

\section{Functors on $\calD$-modules}

\label{sec:Dmods}
If $K$ is a free Abelian group of finite rank, or a finite dimensional vector
space, then we write $K^\dual$ for the dual group or vector space.

We introduce the following notation. Let $X$ be a smooth complex
algebraic variety of dimension $d_X$.  The Abelian category of
algebraic left $\mcd_X$-modules on $X$ is denoted by $\rmM(\mcd_X)$
and the Abelian subcategory of (regular) holonomic $\mcd_X$-modules by
$\rmM_h(\mcd_X)$ (resp. $(\rmM_{rh}(\mcd_X))$. We abbreviate
$\rmD^b(\rmM(\calD_X))$ to $\rmD^b(\calD_X)$, and denote by
$\rmD^b_{h}(\mcd_X)$ (resp.\ $\rmD^b_{rh}(\mcd_X)$) the full
triangulated subcategory in $\rmD^b(\mcd_X)$ consisting of objects
with holonomic (resp.\ regular holonomic) cohomology.

Let $f: X \ra Y$ be a morphism between smooth algebraic varieties
and let $M \in \rmD^b(\mcd_X)$ and $N \in \rmD^b(\mcd_Y)$. The direct
and inverse image functors for $\mcd$-modules are denoted by
\[
f_+ M := \rmR f_*(\mcd_{Y \leftarrow X} \overset{L}{\otimes} M) \quad
\text{and}\quad f^+ M:= \mcd_{X \ra Y} \overset{L}{\otimes} f^{-1}M
     [d_X - d_Y]
\]
respectively.  The functors $f_+$ and $f^+$ preserve (regular) holonomicity
(see e.g., \cite[Theorem 3.2.3]{Hotta}).

\noindent We denote by
\[
\mbd: \rmD^b_h(\mcd_X) \ra (\rmD^b_h(\mcd_X))^{opp}
\]
the holonomic duality functor.  Recall that for a single holonomic
$\mcd_X$-module $M$, the holonomic dual is also a single holonomic
$\mcd_X$-module (\cite[Proposition 3.2.1]{Hotta}) and that holonomic
duality preserves regularity (\cite[Theorem 6.1.10]{Hotta}).

For a morphism $f: X \ra Y$ between smooth algebraic
varieties we additionally define the functors
\[
f_\dag := \mbd \circ
f_+ \circ \mbd\quad \text{ and }\quad f^\dag := \mbd \circ f^+ \circ \mbd.
\]

Let $X$ be an algebraic variety. Denote by $\MHM(X)$ the Abelian
category of algebraic mixed Hodge modules and by $\rmD^b \MHM(X)$ the
corresponding bounded derived category, compare \cite{SaitoMHM,
  SaitoOnMHM}. If $X$ is smooth the forgetful functor to the bounded
derived category of regular holonomic $\mcd_X$-modules is denoted by
\begin{gather}\label{def-Dmod}
\Dmod: \rmD^b \MHM(X) \lra \rmD^b_{rh}(\mcd_X)\, .
\end{gather}
\noindent For each morphism $f:X \ra Y$ between complex algebraic
varieties, there are induced functors
\[
f_*,\,f_!\colon \rmD^b \MHM(X) \lra \rmD^b \MHM(Y)
\]
and
\[
f^*,\, f^!\colon \rmD^b \MHM(Y) \ra \rmD^b \MHM(X)\, ,
\]
which satisfy $\DD\circ f_*=f_!\circ\DD$, $\DD\circ f^*=f^!\circ\DD$,
and which lift the analogous functors $f_+, f_\dag, f^\dag, f^+$ on
$\rmD^b_{rh}(\mcd_X)$ in case $X$ is smooth.

Let $\mbq^H_{pt}$ be the trivial Hodge structure $\mbq$ of type
$(0,0)$, i.e. $\gr^W_i \mbq^H_{pt} = \gr^F_i \mbq^H_{pt} = 0$ for $i
\neq 0$. Viewing it as a Hodge module on a point $pt$, denote by
${^p}\mbq^H_X:= \mbq_X^H[d_X]$ the (mixed) Hodge module $(a^*
\mbq^H_{pt})[d_X]$, where
\[
a: X \ra {pt}
\]
is the unique map to a point.
For smooth $X$ the 
$\mcd$-module underlying ${^p}\mbq_X^H$ is the structure sheaf $\mco_X$ in
cohomological degree zero with $\gr^W_i \mco_X = 0$ for $i \neq d_X$.

Let $j:U \ra X$ be any Zariski dense smooth open subset of $X$ and
let $\mcl$ be a polarizable variation of Hodge structures (that is to
say, a
  vector bundle with a flat connection $\nabla$ such that each fiber
  carries a Hodge structure, with $\nabla(F_p)\subseteq F_{p+1}$ for
  increasing filtrations, and a global polarization pairing) of
weight $w$. Set \[
{^p}\mcl := \mcl \otimes {^p}\mbq^H_U.
\]
We denote by
$\IC_X({^p}\mcl)$ the intersection cohomology complex with coefficients
in ${^p}\mcl$; this is a pure Hodge module of weight $w +d_X$ equal to
$\image (\mch^0 j_!  {^p}\mcl \ra \mch^0 j_* {^p}\mcl)$.  We write
$\IC_X$ for $\IC_X({^p}\mbq^H_U)$; this does not depend on $U$, see
\cite[Thm.\ 5.4.1, p.\ 156]{Di}.


\begin{lem}\label{lem:propinvimage}
Let $(X,\mcs)$ be an algebraic Whitney stratification of $X$ with a
Zariski dense smooth open stratum $U$. Denote by $i_S: S \ra X$ the
embedding of the stratum $S \in \mcs$ in $X$ and let ${^p}\calL$ be as
above. The following holds for
morphisms in $\MHM$:
\begin{enumerate}[(1)]
\item $i^!_S$ is left exact for every $S \in \mcs$ and does not
  decrease weights. (In other words, if $W_{\le k}(\calM)=0$ then
  $W_{\le k}i^!_S(\calM)=0$).
\item\label{propinvimage-2} $\mch^0 i_U^! \IC_X({^p}\mcl) = {^p}\mcl$ and $\mch^k i_U^!
  \IC_X({^p}\mcl) = 0$ for $k \neq 0$.
\item\label{propinvimage-3} $\mch^0 i_S^!\IC_X({^p}\mcl) = 0 $ for $ U \neq S$.
\end{enumerate}
\end{lem}
\begin{proof}
The first statement follows from \cite[Proposition 10.2.11]{KS} and
\cite[(4.5.2)]{SaitoMHM}. The second statement follows from the fact
that $i^!_U = i^*_U$ is just the restriction to the open subset $U$
which is exact. The last point follows from the characterization of
$\IC_X({^p}\mcl)$ as $\image (\mch^0 j_! {^p}\mcl \ra \mch^0 j_* {^p}\mcl)$
and \cite[1.4.22 and 1.4.24]{BBD}.
\end{proof}

\section{Weight filtration on torus embeddings}\label{sec:weightontorus}

\subsection{Basic Notions}\label{subsec-basics}

\begin{ntn}\label{ntn-A}
  If $C$ is a semiring (an additive semigroup closed under
  multiplication) write $CA$ for the $C$-linear combinations of
  the columns of the integer $d\times n$ matrix $A$.
  We assume that $A$ satisfies:
  \begin{enumerate}
  \item $\mbz A = \mbz^d$;\label{A-righttorus}
  \item $A$ is \emph{saturated}: $(\mbr_{\geq 0}A) \cap \mbz^d = \mbn
    A$;\label{A-normal}
  \item $A$ is \emph{pointed}: $\NN A\cap(-\NN A) = \{0_{\ZZ
    A}\}$.\label{A-pointed} 
  \end{enumerate}
  We let $\sigma := \mbr_{\geq 0} A$ be the real cone over $A$ inside
  $\RR^d$ and consider the affine toric variety
  \[
  X:=X_\sigma:= \Spec(\mbc[\mbn A])\subseteq \hatV
  \]
  together with its open dense torus
  \[
  T:= T_\sigma := \Spec(\mbc[\mbz A]).
  \]
  Properties \eqref{A-righttorus}-\eqref{A-pointed} of $A$ above
  imply that $X$ is $d$-dimensional,  
  normal by Hochster's theorem \cite{Hoch}, and has one $T$-fixed point.



Let $ \tau \subseteq \sigma$ be a $d_\tau$-dimensional face of $\sigma$.
We denote by
\[
\tau_\mbz := (\tau+(-\tau)) \cap \mbz^d\;, \qquad \tau_\mbn := \tau
\cap \mbz^d \qquad \text{and} \qquad \tau_\mbr := \tau + (-\tau) =
\textup{span}(\tau)
\]
the $\mbz$-, $\mbn$- and $\mbr$-spans of the collection of $A$-columns
in $\tau$ (considering that $\NN A$ is saturated).
We associate to $\tau$ a $d_\tau$-dimensional torus orbit
\[
T_\tau := \Spec(\mbc[\tau_\mbz])
\]
whose closure in $X=X_\sigma$ via the embedding $i_{\tau,\sigma}\colon
T_\tau\into X_\sigma$ induced by $\CC[\NN A]=\CC[\sigma_\NN]\onto
\CC[\tau_\NN]\into \CC[\tau_\ZZ]$ is
\[
X_\tau :=\Spec(\mbc[\tau_\mbn]){\xleftarrow{ \,\,\phi_\tau\, }\joinrel\rhook} \Spec(\mbc[\tau_\ZZ])=T_\tau.
\]
Saturatedness of $\NN A$  implies that $X_\tau$ is normal.  The variety
\[
U_\tau :=\Spec(\mbc[\sigma_\mbn + \tau_\mbz])
\]
is an open neighborhood of $T_\tau$ in $X$.  The affine toric variety 
\[
X_{\sigma / \tau} := \Spec(\mbc[(\sigma_\mbn + \tau_\mbz) / \tau_\mbz])
\]
with its dense torus 
\[
T_{\sigma/\tau} := \Spec(\mbc[\sigma_\mbz / \tau_\mbz])
\]
is a normal slice to the stratum $T_\tau$: there is a (non-canonical) isomorphism
$U_\tau \simeq X_{\sigma / \tau} \times T_\tau$ and
an inclusion
\[
j_\tau \colon  X_{\sigma / \tau} \times T_\tau \simeq U_\tau\into  X.
\]
The inclusions $T_{\sigma/\tau}\into X_{\sigma/\tau}\into X$
correspond to the (canonical) morphisms $\sigma_\NN\onto
(\sigma_\NN+\tau_\ZZ)/\tau_\ZZ\to \sigma_\ZZ/\tau_\ZZ$.
For any pointed rational polyhedral cone $\rho$ in (a quotient of) $\RR^d$ we
denote by
\[
i_\rho\colon \{\frakx_\rho\}\into X_\rho
\]
the embedding of
the unique torus-fixed point. Then we have the
following commutative diagram
\[
\xymatrix{\{\frakx_{\sigma / \tau}\} \times T_\tau
  \ar[rr]^{i_{\sigma/\tau} \times \id}  \ar[drr]_{i_{\tau,\sigma}} &&
  X_{\sigma / \tau} \times T_\tau \ar[d]_{j_{\tau}} && T_{\sigma /
    \tau} \times T_\tau \simeq T \ar[ll]_{\TXmap_{\sigma / \tau}
    \times \id}  \ar[dll]^\TXmap \\ && X  &&} 
\]
of equivariant maps.
\schluss

\end{ntn}

\medskip


\begin{defn}
  Let $\mu: \GG_m \times Y \ra Y$ be a $\GG_m$-action on the variety $Y$.
  Write $\pr\colon \GG_m\times Y\to Y$ for the projection.
  A holonomic $\calD_{Y}$-module $\calM$ is called
  \emph{$\GG_m$-equivariant} if $\mu^+ \calM
  \simeq \pr^+ \calM$ as $\mcd_{\GG_m\times Y}$-module.
\schluss\end{defn}

If $\boldv\in\ZZ A$ is in the interior $\Int(\sigma^\vee)$ of the dual
cone $\sigma^\vee$, then $\boldv$ defines a $1$-parameter subgroup
\[
\kappa_\boldv\colon \mbg_m = \Spec \mbc[z^\pm] \ra T = \Spec
\mbc[\mbz A]
\]
given by $t^\boldu \mapsto z^{\langle \boldu,\boldv
  \rangle}$.
It extends to a map $\overline{\kappa}_\boldv\colon \mba^1 \ra X_\sigma$
with limit point (recall that $A$ is pointed) equal to the $T$-invariant point 
$\frakx_\sigma \in X_\sigma$. By adjusting the ambient lattice,
similar statements hold for all faces $\tau$.

%
On the level of underlying $D$-modules, the following is
\cite[Prop.~10.4]{Ginsburg-Inv86}.
\begin{lem}\label{lem:equivpoint}
Let $i_\tau: \{\frakx_{\tau}\} \ra X_\tau$ be the inclusion of
the torus fixed point and for any space $X$ denote
\[
a_X:X\to pt
\]
the
projection to a point. If $X=X_\tau$ is one of our orbit closures,
identify $a_{X_\tau}$ with $a_\tau: X_\tau \ra \{\frakx_\tau\}$.

Let $\boldv\in\tau^\vee $ be an integer
element in the relative interior of the dual cone and consider the induced action
of $\mbg_m$ on $X_\tau$. For every $\mbg_m$-equivariant Hodge
module $\mcm$ on
$X_\tau$ we have the following isomorphisms
\begin{align}
{a_\tau}_* \mcm \simeq i_\tau^* \mcm \qquad \text{and} \qquad
{a_\tau}_! \mcm \simeq i_\tau^! \mcm. \notag
\end{align}
\end{lem}
\begin{proof}
  It suffices to consider the case $\tau=\sigma$. 
Denote by $u: X_\sigma \smallsetminus \{\frakx_\sigma\} \ra X_\sigma$ the
open embedding of the complement of the fixed point $\frakx_\sigma$,
abbreviate $i_\sigma$ to $i$ and denote by $a$ the map to a point. We
have the exact triangle
\[
u_! u^{-1} \mcm \lra \mcm \lra i_* i^* \mcm \overset{+1}{\lra}
\]
Applying $a_*$ we get
\[
a_* u_! u^{-1} \mcm \lra a_*\mcm \lra  i^* \mcm \overset{+1}{\lra}
\]
and we will show that $a_*u_!u^{-1}\mcm=0$. 
As $\boldv \in \Int(\sigma^\vee)$ , we have an action 
 $\overline{\kappa}_\boldv\colon  \mba^1 \times X_\sigma \ra X_\sigma$ with
$\overline{\kappa}_\boldv^{-1}(\frakx_\sigma) = (\mba^1 \times \{\frakx_\sigma\} ) \cup (\{0\}
\times X_\sigma)$.  This gives the following Cartesian diagram
\[
\xymatrix{\mbg_m \times (X_{\sigma} \smallsetminus \{\frakx_\sigma\})
  \ar[rr]^{u'} \ar[d]_{\overline{\kappa}_\boldv'} && \mba^1 \times X_\sigma \ar[d]_{\overline\kappa_\boldv}
  \\ X_{\sigma} \smallsetminus \{\frakx_\sigma\} \ar[rr]^u && X_\sigma }
\]
where $\overline{\kappa}_\boldv'$ is the restriction and $u'$ is the
canonical inclusion. Consider
the morphism $g: X_\sigma \ra \mba^1 \times X_\sigma$ with $g(x) =
(1,x)$. The morphism $g$ is a section of $\overline{\kappa}_\boldv$, hence
$\overline{\kappa}_\boldv\circ 
g=\id_{X_\sigma}$. Therefore the composition
\[
a_* \ra a_* {\overline{\kappa}_\boldv}_* \overline{\kappa}_\boldv^*
= (a_{\AA^1\times X_\sigma})_* \overline{\kappa}_\boldv^* 
\ra (a_{\AA^1\times X_\sigma})_* g_* g^*\overline{\kappa}_\boldv^* = a_* g^*\overline{\kappa}_\boldv^*
\]
is the identity transformation. In order to show that $a_* u_!
u^{-1} \mcm = 0$ it is hence enough to prove that the intermediate
module $(a_{\AA^1\times X_\sigma})_* \overline{\kappa}_\boldv^*u_!
u^{-1} \mcm$ vanishes. 

By base change we get the following isomorphism:
\[
(a_{\AA^1\times X_\sigma})_* \overline{\kappa}_\boldv^*u_! u^{-1} \mcm \simeq (a_{\AA^1\times X_\sigma})_*u'_!(\overline{\kappa}_\boldv')^* u^{-1} \mcm.
\]
Since $u^{-1} \mcm$ is $\mbg_m$-equivariant, we have
\[
(\overline{\kappa}_\boldv')^*u^{-1}\mcm \simeq pr^* u^{-1}\mcm \simeq \mbq^H_{\mbg_m}
\boxtimes u^{-1}\mcm. 
\]
Therefore we get
\[
(a_{\AA^1\times X_\sigma})_*u'_!(\overline{\kappa}_\boldv')^* u^{-1} \mcm \simeq
(a_{\AA^1\times X_\sigma})_* u'_!(\mbq^H_{\mbg_m} \boxtimes u^{-1}\mcm)
\simeq (a_{\AA^1\times X_\sigma})_* (u_{1!}  \mbq^H_{\mbg_m} \boxtimes
u_{!}u^{-1}\mcm),
\]
where $u_1 :\mbg_m\ra \mba^1$ is the canonical inclusion. Since
$\HH^\bullet(\mba^1, u_{1!} \mbq^H_{\mbg_m}) = 0$,
    the K\"{u}nneth formula shows that $(a_{\AA^1\times X_\sigma})_*\overline{\kappa}_\boldv^*u_!u^{-1} \mcm = 0$. This
shows the first claim. The second claim follows by dualizing; note
that duals of equivariant modules are equivariant.
\end{proof}

Recall that $\IH(-)$ (and $\IH_c$) denotes intersection cohomology (with
compact support).
\begin{lem}\label{lem:propIHtoric}
Let $\gamma$ be a face of $\sigma$, $X_\gamma$ the associated
$d_\gamma$-dimensional affine toric variety. The following holds
\begin{enumerate}
\item $\IH^{d_\gamma+ k}_c(X_{\gamma}) \simeq \left(\IH^{d_\gamma
  -k}(X_{\gamma})(d_\gamma)\right)^\dual$.
\item $\IH^{d_\gamma+k}(X_\gamma) =\IH_c^{d_\gamma-k}(X_\gamma) = 0$ for $k \geq 0$.
\item $\IH^k(X_\gamma) = 0$ for $k$ odd.
\item $\IH^{2k}(X_\gamma)$ and $\IH^{2k}_c(X_\gamma)$ are pure Hodge
  structures of Hodge--Tate type with weight $2k$, \emph{i.e.}
\[  
 gr^W_{i}\IH^{2k}(X_\gamma) = 0 \quad \text{and}\quad gr^F_{j} \IH^{2k}(X_\gamma) = 0  
 \text{ for $i \neq 2k$ and $j \neq -k$.}
 \]
\end{enumerate}
\end{lem}
\begin{proof} Temporarily, write $a$ for $a_\gamma$ and $i$ for $i_\gamma$.
Claim (1) follows from Verdier duality:
\begin{eqnarray*}
\IH^{d_\gamma+ k}_c(X_{\gamma})
&\simeq&  \HH^k a_!\IC_{X_{\gamma}}( {^p}\mbq^H_{T_{\gamma}}) \\
&\simeq& \HH^k \mbd a_* \mbd\IC_{X_{\gamma}}( {^p}\mbq^H_{T_{\gamma}})\\
&\simeq& \HH^k \mbd a_*(\IC_{X_{\gamma}}( {^p}\mbq^H_{T_{\gamma}})(d_\gamma)) \\
&\simeq& \left( \HH^{-k}a_*(\IC_{X_{\gamma}}( {^p}\mbq^H_{T_{\gamma}})(d_\gamma)) \right)^\dual \\
&\simeq& \left(\IH^{d -k}(X_{\gamma})(d_\gamma)\right)^\dual.
\end{eqnarray*}

From Lemma \ref{lem:equivpoint} we have the isomorphisms
\[
\IH^k(X_\gamma) = \HH^k a_*\IC_{X_\gamma}(\mbq^H_{T_\gamma}) \simeq \HH^k
  i^*\IC_{X_\gamma}(\mbq^H_{T_\gamma}) =
  \HH^k\IC_{X_\gamma}(\mbq^H_{T_\gamma})_{\frakx_\gamma}.
\]
  Claim (2) follows from \cite[Theorem
  1.2]{Fies}, which also implies---in conjunction with Remark
  ii.) in
  loc. cit.---Claim (3).
Claim (4) follows from
\cite[Corollary 4.12]{Weber}.
\end{proof}

Let now $\tau, \gamma$ be faces of $\sigma$ with $\tau \subseteq \gamma$
and set
\[
X_{\gamma / \tau} := \Spec(\mbc[(\gamma_\mbn  + \tau_\mbz)  /\tau_\mbz ]).
\]

The following result discusses (derived) pullbacks of constant variations of
Hodge structures to torus orbits.

\begin{lem}\label{lem:weightonorbitIC}
Let $i_{\tau,\gamma}: T_\tau \to X_\tau\to  X_\gamma$ be the torus orbit
embedding and let $H$ be a polarizable Hodge structure of weight
$w$ (on a point). Then ${^p}\mcl:= H\otimes {^p}\mbq^H_{T_\gamma}$ is a
(constant) variation of polarizable Hodge structures of weight
$w+d_\gamma$ on $T_\gamma$. We have the following isomorphisms in
$\MHM(T_\tau)$:
\begin{align}
  \mch^k(i_{\tau,\gamma}^!\,\IC_{X_{\gamma}}({^p}\mcl))&\simeq H
  \otimes\IH_c^{d_\gamma - d_\tau +k}(X_{\gamma / \tau}) \otimes
            {^p}\mbq^H_{T_\tau}, \\
 \label{eqn-weightonorbit-2}
 \mch^k(i_{\tau,\gamma}^*\,\IC_{X_{\gamma}}({^p}\mcl))&\simeq
 H \otimes \IH^{d_\gamma - d_\tau +k}(X_{\gamma / \tau})
 \otimes {^p}\mbq^H_{T_\tau}.
\end{align}
The weight filtration satisfies:
\[
\gr^W_j \mch^k(i_{\tau,\gamma}^!\IC_{X_\gamma}({^p}\mcl)) =
 \gr^W_j \mch^k(i_{\tau,\gamma}^*\IC_{X_\gamma}({^p}\mcl)) = 0
\]
for $j \neq w+d_\gamma+k$.
\end{lem}

\begin{proof}
  Consider the following diagram:
 \[
\xymatrix{\{\frakx_{\gamma / \tau}\} \ar[rr]^{i_{\gamma/\tau}}  && X_{\gamma / \tau} &&
  T_{\gamma / \tau} \ar[ll]_{\TXmap_{\gamma / \tau}} \\
  \{\frakx_{\gamma / \tau}\} \times T_\tau \ar[rr]^{i_{\gamma/\tau} \times \id}
  \ar[u]^{p_1} \ar[drr]_{i_{\tau,\gamma}} && X_{\gamma / \tau} \times T_\tau
  \ar[u]^{p_2} \ar[d]_{j_{\tau}^\gamma} && T_{\gamma / \tau} \times
  T_\tau\simeq T_\gamma
  \ar[ll]_{\TXmap_{\gamma / \tau} \times \id} \ar[u]^{p_3} \ar[dll]^{\TXmap_\gamma} \\
  && X_\gamma &&}. \notag
\]
Since ${^p}\mcl$ is a constant variation of Hodge structures on
$T_\gamma \simeq T_{\gamma / \tau} \times T_\tau$, we have
(cf. \cite[(4.4.2)]{SaitoMHM})
%
\begin{align}
{^p}\mcl &= H\otimes \mbq^H_{T_\gamma}[d_\gamma] \simeq (H \otimes
\mbq^H_{T_{\gamma / \tau}}(-d_\tau)[-2 d_\tau +d_\gamma]) \boxtimes
\mbq^H_{T_\tau}(d_\tau)[2 d_\tau] \notag \\ 
&\simeq p_3^!(H \otimes {^p}\mbq^H_{T_{\gamma / \tau}}(-d_\tau)[-d_\tau])
\simeq p_3^!\; ({^p}\tilde{\mcl}(-d_\tau)[-d_\tau]) \notag 
\end{align}
where we have set ${^p}\tilde{\mcl} = H \otimes {^p}\mbq^H_{T_{\gamma/ \tau }}$.
We have the following isomorphisms
\begin{align}
i^!_{\tau,\gamma}\IC_{X_{\gamma}}({^p}\mcl) &\simeq (i_{\gamma/\tau} \times \id)^!
(j_\tau^\gamma)^!\IC_{X_{\gamma}}({^p}\mcl)\notag \\ 
 &\simeq (i_{\gamma/\tau} \times \id)^!\IC_{X_{\gamma / \tau} \times
  T_\tau}({^p}\mcl) \notag \\ 
&\simeq (i_{\gamma/\tau} \times \id)^!  p_2^!\IC_{X_{\gamma
    /\tau}}({^p}\tilde{\mcl})(-d_\tau)[-d_\tau] \notag \\ 
&\simeq p_1^!\, i_{\gamma/\tau}^!\,\IC_{X_{\gamma
    /\tau}}({^p}\tilde{\mcl})(-d_\tau)[-d_\tau] \notag \\ 
&\simeq i_{\gamma/\tau}^!\,\IC_{X_{\gamma / \tau}}({^p}\tilde{\mcl})\boxtimes
\mbd\mbq^H_{T_\tau}(-d_\tau)[-d_\tau] \notag \\ 
&\simeq i_{\gamma/\tau}^!\,\IC_{X_{\gamma/ \tau}}({^p}\tilde{\mcl})\boxtimes
\mbq^H_{T_\tau}[d_\tau] \notag \\ 
&\simeq i_{\gamma/\tau}^!\,\IC_{X_{\gamma /\tau}}({^p}\tilde{\mcl})\boxtimes
    {^p}\mbq^H_{T_\tau}. \notag 
\end{align}

Since $\IC_{X_{\gamma / \tau}}({^p}\tilde{\mcl})$ is $T_{\gamma /
  \tau}$-equivariant it follows from Lemma \ref{lem:equivpoint} that
\begin{align*}
\HH^k(i_{\gamma/\tau}^!\,\IC_{X_{\gamma/ \tau}}({^p}\tilde{\mcl})) &\simeq \HH^k(a_!\,\IC_{X_{\gamma / \tau}}({^p}\tilde{\mcl})) \\
&\simeq \HH^k(a_!\,\IC_{X_{\gamma / \tau}}(H \otimes {^p}\mbq^H_{T_{\gamma / \tau}})) \\
&\simeq \IH^{d_\gamma - d_\tau +k}_c(X_{\gamma / \tau}) \otimes H  \notag 
\end{align*}
as mixed Hodge structures. This gives the isomorphism
\[
\mch^k(i^!_{\tau,\gamma}\IC_{X_\gamma}({}^p\mcl))\simeq H
\otimes
\IH_c^{d_\gamma - d_\tau+k}(X_{\gamma /\tau})
\otimes {^p}\mbq^H_{T_\tau}\, .
\]

The weight filtration $\{W_k\IH^i(X_{\gamma/\tau})\}_k$ on the
intersection cohomology of $X_{\gamma / \tau}$ satisfies $\gr^W_k
\IH^i(X_{\gamma /\tau}) = 0$ if $i \neq k$.
Hence we get
\[
\gr^W_i \mch^k(i^!_{\tau,\gamma}\IC_{X_\gamma}(\mcl)) =
\bigoplus_{i=l_1+l_2+l_3}\gr^W_{l_1} H \otimes_\mbc
\gr^W_{l_2}\IH_c^{d_\gamma - d_\tau + k}(X_{\gamma / \tau})
\otimes_\mbc \gr^W_{l_3} {^p}\mbq^H_{T_\tau} = 0
\]
for $i \neq w +\left(d_\gamma - d_\tau +k\right) + d_\tau= w+ d_\gamma
+ k$.

The statement \eqref{eqn-weightonorbit-2} follows from a dual proof.
\end{proof}

\subsection{A recursion}

The torus orbits $T_\tau \subseteq X$ equip $X^{an}$ with a Whitney
stratification (cf. \cite[Proposition 1.14]{DimcaHypersurface}. Since
the morphisms $\TVmap,\TXmap$ from \eqref{map-factorization} are
affine, algebraic, and stratified, the perverse sheaf underlying
$\TXmap_* {^p}\mbq^H_T$ is constructible with respect to this
stratification. Since $\TXmap_* {^p}\mbq^H_T$ is a mixed Hodge module
its
weight graded parts are direct sums of intersection complexes (with
possibly twisted coefficients) having support on the orbit closures
$X_\tau = \overline{T}_\tau$. We write
\begin{eqnarray}\label{eqn-weight-gr}
\gr^W_k \TXmap_* {^p}\mbq^H_T = \bigoplus_\gamma\IC_{X_\gamma}(
   {^p}\mcv_{(\gamma,k)}).
\end{eqnarray}
Here the direct sum is understood as a direct sum over all faces
$\gamma$ of $\sigma$, and ${^p}\mcv_{(\gamma,k)}$ is a polarizable
variation of Hodge structures of weight $k$ on $T_\gamma$.

Here and elsewhere, for a mixed Hodge module $\calM$ on $Y$, we regard
  as equivalent via Kashiwara equivalence, for $Y$ closed in $Y'$,
  $\IC_Y(\calM)$ and its direct image on $Y'$, without necessarily
  explicitly referencing $Y'$. Moreover, we say that \emph{$\calM$
    has weight $\geq k$} if $\gr^W_{\ell}\calM=0$ for $\ell< k$. 

Our first result on \eqref{eqn-weight-gr} is a recursive formula; we
continue to denote $d_\sigma$ by just $d$ and $X_\sigma$ by just
$X$:
  
\begin{prop}\label{prop:ICcomp}
The weight filtration on the mixed Hodge module $\TXmap_*{^p}\mbq^H_T$
satisfies the following properties.
\begin{enumerate}[(1)]
\item \label{ICcomp-1}$\gr^W_{d+e} \TXmap_*{^p}\mbq^H_T = 0$ for $e \neq 0, 1, \ldots, d$.
\item \label{ICcomp-2}$\supp \gr^W_{d+e} \TXmap_*{^p}\mbq^H_T \subseteq \bigcup_{d_\gamma \leq d-e} X_\gamma$.
\item \label{ICcomp-3}$\gr^W_d \TXmap_*{^p}\mbq^H_T=W_d\, \TXmap_* {^p}\mbq^H_T
  =\IC_{X_\sigma}$, by which we denote $\IC_X({}^p\mbq^H_T)$.
\item \label{ICcomp-4}$\gr^W_{d+1} \TXmap_* {^p}\mbq^H_T =
  \bigoplus_\tau\IC_{X_\tau}(L^0_{(\tau,d+1)} \otimes
           {^p}\mbq^H_{T_\tau})$ where
           \[ L^k_{(\tau,d+1)} := \IH^{d_\sigma-d_\tau+k+1}_c(X_{\sigma/\tau})
           \]
           is the intersection homology group with compact support of
           $X_{\sigma/\tau}$. 
\item \label{ICcomp-5}For $e> 1$, $\gr^W_{d+e} \TXmap_* {^p}\mbq^H_T = \bigoplus_\tau\IC_{X_\tau}(L^0_{(\tau,d+e)}\otimes {^p}\mbq^H_{T_\tau})$ where
\begin{eqnarray}\label{eqn-L-recursion}
L^k_{(\tau,d+e)} \simeq \frac{\left( \underset{\gamma \supseteq \tau
  }\bigoplus L^0_{(\gamma,d+e-1)}
  \otimes\IH_c^{d_\gamma-d_\tau+k+1}(X_{\gamma / \tau})
  \right)}{L^{k+1}_{\tau,d+e-1}}.
\end{eqnarray}
\end{enumerate}
An essential feature of the situation is that:
\begin{enumerate}[(1)]\setcounter{enumi}{5}
  \item\label{ICcomp-6}
     for all $e$, the module
    \[
      {^p}\calL^k_{(\tau,d+e)}:=\calH^ki^!_{\tau,\sigma}(\TXmap_*{^p}\QQ^H_{T_\sigma}/W_{d+e-1}\TXmap_*{^p}\QQ^H_{T_\sigma})
      \]
      is pure of weight $d+e+k$. It is zero for $d_\tau\geq d-e+1-k$
      and in any case isomorphic to a finite sum of copies of
    ${^p}\mbq^H_{T_\tau}$.
\end{enumerate}
\end{prop}

\begin{proof}
In order to ease the notation we denote in this proof by $\mbq_\gamma$
the Hodge module ${^p}\mbq^H_{T_\gamma}$.

We will proceed by induction on $e$. Obviously, for $e\ll 0$, all
parts hold trivially. Assuming Property \eqref{ICcomp-6} up to $e$ as
well as Property \eqref{ICcomp-5} up to $e-1$, we show Property
\eqref{ICcomp-6} for $e+1$ and Property \eqref{ICcomp-5} for
$e$. Property \eqref{ICcomp-2} is then a direct
consequence. Properties \eqref{ICcomp-3} and \eqref{ICcomp-4} are the
induction start and are proved in the same fashion as the induction
step, but look less uniform.

Since $\mbq_\sigma$ has weight $d$ and direct images do not decrease
weights, the direct image $\TXmap_*\mbq_\sigma = \mch^0 \TXmap_*\mbq_\sigma$ has
weight $\geq d$. Property \eqref{ICcomp-1} is hence a consequence of
Property \eqref{ICcomp-2}.

We make the following Ansatz for the part
of $\TXmap_*\mbq_\sigma$ of weight $d$:
\[
W_d \TXmap_* \mbq_\sigma = \bigoplus_\gamma\IC_{X_\gamma}({^p}\mcv_{(\gamma,d)}),
\]
where ${^p}\mcv_{(\gamma,d)}$ is a polarizable variation of Hodge
structures on $T_\gamma$ of weight $d$.

We begin with the lowest weight case $e=0$; then \eqref{ICcomp-5} is
vacuous. Consider the exact
sequence
\begin{equation}\label{eq:exseqwd}
0 \lra \bigoplus_\gamma\IC_{X_\gamma}({^p}\mcv_{(\gamma,d)}) \lra \TXmap_* \mbq_\sigma  \lra \TXmap_* \mbq_\sigma / W_d \TXmap_* \mbq_\sigma \lra 0.
\end{equation}
Let $\tau$ be an $d_\tau$-dimensional face of $\sigma$;
$i_{\tau,\sigma}: T_{\tau} \ra X$ 
is the natural embedding. Apply the functor $i^!_{\tau,\sigma}$ to
\eqref{eq:exseqwd}, recalling that it is left exact and does not
decrease weights (\emph{cf}. Lemma
\ref{lem:propinvimage}.(1)). Because of Lemma
\ref{lem:propinvimage}.(2),(3) we get a long exact sequence
\[
0 \lra {^p}\mcv_{(\tau,d)} \lra \mch^0 i_{\tau,\sigma}^! \TXmap_* \mbq_\sigma \lra
\mch^0 i^!_{\tau,\sigma}\left( \TXmap_* \mbq_\sigma / W_d \TXmap_* \mbq_\sigma \right)
\lra \cdots.
\]
Since $i^!_{\tau,\sigma} \TXmap_* \mbq_\sigma = 0$ for $d_\tau < d$ by
Lemma \ref{lem:propinvimage}.\eqref{propinvimage-3} we
obtain
Property \eqref{ICcomp-6} for $e=0$ and find
${^p}\mcv_{(\tau,d)} ={^p}\calL^k_{(\tau,d)}= 0$ for those $\tau$. 
In the case $d_\tau=d$,
we have $\tau = \sigma$ and $i^!_{\sigma,\sigma} \TXmap_* \mbq_\sigma =
\mbq_\sigma$  and therefore obtain
\[
0 \lra {^p}\mcv_{(\sigma,d)} \lra \mbq_\sigma \lra \mch^0
i^!_{\sigma,\sigma}\left( \TXmap_* \mbq_\sigma / W_d \TXmap_* \mbq_\sigma \right).
\]
Since $\TXmap_* \mbq_\sigma / W_d \TXmap_* \mbq_\sigma$ has weight $> d$ and
$i^!_{\sigma,\sigma}$ does not decrease weight,
${^p}\mcv_{(\sigma,d)}\simeq \mbq_\sigma$. Altogether we have:
\[
{^p}\mcv_{(\tau,d)} = \begin{cases}\mbq_\sigma & \text{for}\; d_\tau=d,
  \\ 0 & \text{for}\; d_\tau < d. \end{cases}
\]
Thus, for all $\tau$, ${^p} \mcl^0_{(\tau,d)} = {^p}\mcv_{(\tau,d)}$  and 
${^p}\calL^{\neq0}_{(\tau,d)}$ vanishes.
This shows Property \eqref{ICcomp-3} (and embodies Property
\eqref{ICcomp-6} for $e=0$).

We next consider the weight $d+1$ part. To begin, we use the fact that
$W_d \TXmap_* \mbq_\sigma =\IC_{X_\sigma}$ in order to compute
$\mch^k(i_{\tau,\sigma}^! \left( \TXmap_* \mbq_\sigma / W_d \TXmap_*
\mbq_\sigma\right))$ for each face $\tau$ and all $k \geq 0$. The exact
sequence \eqref{eq:exseqwd} becomes
\[
0 \lra\IC_{X_\sigma} \lra \TXmap_* \mbq_\sigma \lra \TXmap_* \mbq_\sigma / W_d \TXmap_* \mbq_\sigma \lra 0.
\]
We apply again $i_{\tau,\sigma}^!$ to this short exact sequence and obtain
\[
\xymatrix{0 \ar[r] & \mch^0 i_{\tau,\sigma}^!\IC_{X_\sigma} \ar[r] & \mch^0
  i_{\tau,\sigma}^! \TXmap_* \mbq_\sigma \ar[r] & \mch^0
  i_{\tau,\sigma}^! \left( \TXmap_* 
  \mbq_\sigma / W_d \TXmap_* \mbq_\sigma \right) \ar[lld] \\
  & \mch^1 i_{\tau,\sigma}^!\IC_{X_\sigma} \ar[r] & \mch^1 i_{\tau,\sigma}^! \TXmap_*
  \mbq_\sigma \ar[r] & \mch^1 i_{\tau,\sigma}^! \left( \TXmap_* \mbq_\sigma / W_d
  \TXmap_*\mbq_\sigma \right) \ar[lld] \\   
  & \mch^2 i_{\tau,\sigma}^!\IC_{X_\sigma} \ar[r] & \cdots} .
\]

For $d_\tau = d$ we have $\mch^k i_{\sigma,\sigma}^! \TXmap_* \mbq_\sigma = \mch^k
i_{\sigma,\sigma}^!\IC_{X_\sigma} = 0$ for $k \geq 1$ (as $i_{\sigma,\sigma}$ is an open
embedding and therefore $i^!_{\sigma,\sigma}$ is exact) and so
\[
  {{^p}}\mcl^k_{(\sigma,d+1)} =\mch^k i^!_{\sigma,\sigma} (\TXmap_* \mbq_\sigma / W_d
  \TXmap_* \mbq_\sigma)
  \]
vanishes for all $k$ (the case $k= 0$ follows from
  ${^p}\mcl^0_{(\sigma,d)} = \mbq_\sigma$).

For $d_\tau < d$ we have $\mch^k i^!_{\tau,\sigma} \TXmap_* \mbq_\sigma = 0$ for
all $k$, hence by Lemma \ref{lem:weightonorbitIC} we have
\begin{equation}\label{eqn-Ld+1}
  {{^p}}\mcl^k_{(\tau,d+1)} :=
  \mch^k i^!_{\tau,\sigma} \left( \TXmap_* \mbq_\sigma /W_d \TXmap_* \mbq_\sigma\right)
  \simeq \mch^{k+1} i^!_{\tau,\sigma}\IC_{X_\sigma}
  \simeq\IH_c^{d_\sigma- d_{\tau} +k+1}(X_{\sigma / \tau}) \otimes \mbq_{\tau}.
\end{equation}
In particular, ${^p}\calL^k_{(\tau,d+1)}=L^k_{(\tau,d+1)}\otimes
\QQ_\tau$ with
$L^k_{(\tau,d+1)}=\IH_c^{d_\sigma-d_\tau+k+1}(X_{\sigma/\tau})$.
Since $\IH_c^{k+1}(X_{\sigma / \sigma}) = 0$ for all $k \geq 0$,
formula \eqref{eqn-Ld+1} is also correct for $\tau = \sigma$. Notice that ${^p}
\mcl^k_{(\tau,d+1)}$ is pure of weight $d+1+k$ and since
$\IH_c^{d_\sigma -d_\tau +k+1}(X_{\sigma /\tau})=0$ for $d_\sigma -
d_\tau + k +1 > 2(d_\sigma - d_\tau)$ we have ${^p}\mcl^k_{(\tau,d+1)} =
0$ for $d_\tau \geq d-k$. This shows Property \eqref{ICcomp-6} in the
case $e=1$.

In order to compute the weight $(d+1)$ part of $\TXmap_* \mbq_\sigma$ we make
the Ansatz
\[
\gr^W_{d+1}\TXmap_* \mbq_\sigma =
\bigoplus_\gamma\IC_{X_\gamma}({^p}\mcv_{(\gamma,d+1)})
\]
and consider
the exact sequence
\begin{eqnarray}\label{eq-ses-weight-d+1}
0 \lra \bigoplus_\gamma\IC_{X_\gamma}({^p}\mcv_{(\gamma,d+1)}) \lra
\TXmap_* \mbq_\sigma / W_d \TXmap_* \mbq_\sigma \lra \TXmap_* \mbq_\sigma / W_{d+1}
\TXmap_* \mbq_\sigma \lra 0.
\end{eqnarray}
The functor $i_{\tau,\sigma}^!$ produces the long exact cohomology sequence
\[
0 \lra {^p}\mcv_{(\tau,d+1)} \lra {^p}\mcl^0_{(\tau,d+1)} \lra
\mch^0i_{\tau.\sigma}^! \left( \TXmap_* \mbq_\sigma / W_{d+1} \TXmap_* \mbq_\sigma
\right) \lra \cdots.
\]
Since $i^!_{\tau,\sigma}$ does not decrease weight, the third term has weight
$d+2$ or more, and since ${^p}\mcl^0_{(\tau,d+1)}$ is pure of weight $d+1$,
\eqref{eqn-Ld+1} yields
\[
{^p}\mcv_{(\tau,d+1)} = {^p}\mcl^0_{(\tau,d+1)} =\IH_c^{d_\sigma-d_\tau +1}(X_{\sigma/ \tau}) \otimes \mbq_\tau,
\]
which shows Property \eqref{ICcomp-4}.

With this, \eqref{eq-ses-weight-d+1} becomes now 
\[
0 \lra \bigoplus_\gamma\IC_{X_\gamma}({^p}\mcl^0_{(\gamma,d+1)}) \lra
\TXmap_* \mbq_\sigma / W_d \TXmap_* \mbq_\sigma \lra \TXmap_*
\mbq_\sigma / W_{d+1} \TXmap_* \mbq_\sigma \lra 0,
\]
and applying $i_{\tau,\sigma}^!$ we obtain
\[
\xymatrix{
  0 \ar[r] & L^0_{(\tau,d+1)} \otimes \mbq_\tau = \underset{\gamma \supseteq \tau}\bigoplus \left( L^0_{(\gamma,d+1)} \otimes\IH_c^{d_\gamma-d_\tau}(X_{\gamma / \tau}) \otimes \mbq_\tau \right) \ar[r] &L^0_{(\tau,d+1)} \otimes \mbq_\tau \ar[r] &\mch^0 i_{\tau,\sigma}^!\left(\TXmap_* \mbq_\sigma / W_{d+1} \TXmap_* \mbq_\sigma \right) \ar[lld] \\ 
& \underset{\gamma \supseteq \tau}\bigoplus \left( L^0_{(\gamma,d+1)} \otimes\IH_c^{d_\gamma-d_\tau+1}(X_{\gamma/\tau}) \otimes \mbq_\tau \right) \ar[r] & L^1_{(\tau,d+1)}\otimes \mbq_\tau \ar[r] & \mch^1 i_{\tau,\sigma}^!\left(\TXmap_* \mbq_\sigma / W_{d+1} \TXmap_* \mbq_\sigma \right) \ar[lld] \\
& \underset{\gamma \supseteq \tau}\bigoplus \left( L^0_{(\gamma,d+1)} \otimes\IH_c^{d_\gamma-d_\tau+2}(X_{\gamma/\tau}) \otimes \mbq_\tau \right) \ar[r] & \cdots.}
\]
Here, the first column is owed to Lemma
\ref{lem:weightonorbitIC} and the equality in the first row follows from Lemma \ref{lem:propinvimage}(3) resp. Lemma \ref{lem:propIHtoric} (2).

Since both $\bigoplus_{\gamma \supseteq \tau} \left( L^0_{(\gamma,d+1)}
\otimes\IH_c^{d_\gamma-d_\tau+k}(X_{\gamma / \tau}) \otimes \mbq_\tau
\right)$ and $L^k_{(\tau,d+1)} \otimes \mbq_\tau$ are pure of weight
$d+1+k$, and since $\mch^k i^!_{\tau,\sigma} ( \TXmap_* \mbq_\sigma / W_{d+1} \TXmap_*
\mbq_\sigma)$ has weight $>(d+1+k)$ the long exact sequence splits
into sequences
\[
0 \lra \mch^k i^!_{\tau,\sigma} ( \TXmap_* \mbq_\sigma / W_{d+1} \TXmap_* \mbq_\sigma)
\lra \underset{\gamma \supseteq \tau}\bigoplus \left( L^0_{(\gamma,d+1)}
\otimes\IH_c^{d_\gamma-d_\tau+k+1}(X_{\gamma / \tau}) \otimes
\mbq_\tau \right) \lra L^{k+1}_{(\tau,d+1)} \otimes \mbq_\tau \lra 0, 
\]
pure of weight $d+1+k+1$. The category of pure Hodge modules is
semisimple and so there is a (non-canonical) splitting
which induces an identification
\begin{gather*}
\mch^k i^!_{\tau,\sigma} ( \TXmap_* \mbq_\sigma / W_{d+1} \TXmap_* \mbq_\sigma) \,\,\,\simeq\,\,\, \frac{\left( \underset{\gamma \supseteq \tau }\bigoplus L^0_{(\gamma,d+1)} \otimes\IH_c^{d_\gamma-d_\tau+k+1}(X_{\gamma / \tau}) \right)}{L^{k+1}_{(\tau,d+1)}} \otimes \mbq_\tau 
\end{gather*}
as pure Hodge modules.
We now define vector spaces ${^p}\mcl^k_{(\tau,d+2)}$ by
\[
{^p}\mcl^k_{(\tau,d+2)} := L^k_{(\tau,d+2)} \otimes \mbq_\tau := \mch^k i^!_{\tau,\sigma} ( \TXmap_* \mbq_\sigma / W_{d+1} \TXmap_* \mbq_\sigma), 
\]
a pure Hodge module of weight $d+2+k$. Since $L^0_{(\gamma,d+1)}$ is
zero for $d_\gamma \geq d$ and $\IH_c^{d_\gamma-d_\tau+k+1}(X_{\gamma
  / \tau})$ is zero for $d_\gamma - d_\tau +k+1 > 2(d_\gamma -
d_\tau)$, the term ${^p}\mcl^{k}_{(\tau,d+2)}$ is zero for $d_\tau
\geq d-1-k$; this proves Property \eqref{ICcomp-6} for $e=2$.

\medskip

We will now provide the inductive step, much in parallel to the
above. Assume that
\[
{^p}\mcl^k_{(\tau,d+e)} = L^k_{(\tau,d+e)} \otimes \mbq_\tau = \mch^k i^!_{\tau,\sigma} ( \TXmap_* \mbq_\sigma / W_{d+e-1} \TXmap_* \mbq_\sigma) 
\]
is pure of weight $d+e+k$ and ${^p}\mcl^{k}_{(\tau,d+e)} = 0$ for
$d_\tau \geq d-e+1-k$ (\emph{i.e.}, Property 6 at level $e$).

In order to compute the weight $d+e$ part of $\TXmap_* \mbq_\sigma$ we make
the Ansatz
\[
\gr^W_{d+e} \TXmap_* \mbq_\sigma =
\bigoplus_\gamma\IC_{X_\gamma}({^p}\mcv_{(\gamma,d+e)})
\]
and consider
the exact sequence
\begin{equation}\label{eqn-ansatzd+e}
0 \lra \bigoplus_\gamma\IC_{X_\gamma}({^p}\mcv_{(\gamma,d+e)}) \lra
\TXmap_* \mbq_\sigma / W_{d+e-1} \TXmap_* \mbq_\sigma \lra \TXmap_* \mbq_\sigma /
W_{d+e} \TXmap_* \mbq_\sigma \lra 0
\end{equation}
We apply the functor $i_{\tau,\sigma}^!$ and get the long exact cohomology
sequence
\[
0 \lra {^p}\mcv_{(\tau,d+e)} \lra {^p}\mcl^0_{(\tau,d+e)} \lra
\mch^0i_{\tau,\sigma}^! \left( \TXmap_* \mbq_\sigma / W_{d+e} \TXmap_* \mbq_\sigma
\right) \lra \cdots.
\]
Since $i^!_{\tau,\sigma}$ does not decrease weight, the third term has weight
greater than $(d+e)$, and as ${^p}\mcl^0_{(\tau,d+e)}$ is pure of weight $d+e$
we find
\[
{^p}\mcv_{(\tau,d+e)} = {^p}\mcl^0_{(\tau,d+e)}.
\]
The exact sequence \eqref{eqn-ansatzd+e} now becomes
\[
0 \lra \bigoplus_\gamma\IC_{X_\gamma}({^p}\mcl^0_{(\gamma,d+e)}) \lra \TXmap_* \mbq_\sigma / W_{d+e-1} \TXmap_* \mbq_\sigma \lra \TXmap_* \mbq_\sigma / W_{d+e} \TXmap_* \mbq_\sigma \lra 0,
\]
and $i_{\tau,\sigma}^!$ induces
\[
\xymatrix{0 \ar[r] & L^0_{(\tau,d+e)} \otimes \mbq_\tau  \ar[r] &L^0_{\tau,d+e} \otimes \mbq_\tau \ar[r] &\mch^0 i_{\tau,\sigma}^!\left(\TXmap_* \mbq_\sigma / W_{d+e} \TXmap_* \mbq_\sigma \right) \ar[lld] \\ 
& \underset{\gamma \supseteq \tau}\bigoplus \left( L^0_{(\gamma,d+e)} \otimes\IH_c^{d_\gamma-d_\tau+1}(X_{\gamma/\tau}) \otimes \mbq_\tau \right) \ar[r] & L^1_{\tau,d+e}\otimes \mbq_\tau \ar[r] & \mch^1 i_{\tau,\sigma}^!\left(\TXmap_* \mbq_\sigma / W_{d+e} \TXmap_* \mbq_\sigma \right) \ar[lld] \\
& \underset{\gamma \supseteq \tau}\bigoplus \left( L^0_{(\gamma,d+e)} \otimes\IH_c^{d_\gamma-d_\tau+2}(X_{\gamma/\tau}) \otimes \mbq_\tau \right) \ar[r] & \cdots.}
\]
Since both $\bigoplus_{\gamma \supseteq \tau} \left(
L^0_{(\gamma,d+e)} \otimes\IH_c^{d_\gamma-d_\tau+k}(X_{\gamma / \tau})
\otimes \mbq_\tau \right)$ and $L^k_{\tau,d+e} \otimes \mbq_\tau$ are
pure of weight $d+e+k$, and since furthermore $\mch^k
i_{\tau,\sigma}^! ( \TXmap_* \mbq_\sigma / W_{d+e} \TXmap_*
\mbq_\sigma)$ has weight greater than $ (d+e+k)$, the long exact
sequence splits in $\MHM(X_\sigma)$ into sequences
\[
0 \lra \mch^k i^!_{\tau,\sigma} ( \TXmap_* \mbq_\sigma / W_{d+e} \TXmap_* \mbq_\sigma) \lra \underset{\gamma \supseteq \tau}\bigoplus \left( L^0_{(\gamma,d+e)} \otimes\IH_c^{d_\gamma-d_\tau+k+1}(X_{\gamma / \tau}) \otimes \mbq_\tau \right) \lra L^{k+1}_{(\tau,d+e)} \otimes \mbq_\tau \lra 0.
\]
The center term is pure of weight $d+e+k+1$; hence the
outer terms are as well. Since the category of pure
Hodge modules is semisimple, the sequence splits (non-canonically) and
there is an
identification
\[
\mch^k i^!_{\tau,\sigma} ( \TXmap_* \mbq_\sigma / W_{d+e} \TXmap_* \mbq_\sigma) \,\,\,\simeq\,\,\, \frac{\left( \underset{\gamma \supseteq \tau }\bigoplus L^0_{(\gamma,d+e)} \otimes\IH_c^{d_\gamma-d_\tau+k+1}(X_{\gamma / \tau}) \right)}{L^{k+1}_{(\tau,d+e)}} \otimes \mbq_\tau. 
\]
We now define $ {^p}\mcl^k_{(\tau,d+e+1)}$ and $L^k_{(\tau,d+e+1)}$ by
\[
{^p}\mcl^k_{(\tau,d+e+1)} := L^k_{(\tau,d+e+1)} \otimes \mbq_\tau := \mch^k i^!_{\tau,\sigma} ( \TXmap_* \mbq_\sigma / W_{d+e} \TXmap_* \mbq_\sigma), 
\]
and reiterate that ${^p}\mcl^k_{(\tau,d+e+1)}$ is pure of weight
$d+e+1+k$. Since $L^0_{(\gamma,d+e)}$ is zero for $d_\gamma \geq
d-e+1$ and $\IH_c^{d_\gamma-d_\tau+k+1}(X_{\gamma /\tau})$ is zero for
$d_\gamma - d_\tau +k  + 1 > 2(d_\gamma - d_\tau)$, the term
${^p}\mcl^{k}_{(\tau,d+e+1)}$
vanishes for $d_\tau \geq d-e+1-k$.  This finishes the inductive step
for Property 
\eqref{ICcomp-6}, establishes \eqref{ICcomp-5} in the process, and  hence
completes the proof.
%
\end{proof}
\begin{rem}
  It has been pointed out by A.~L\"orincz to us that the constancy of
  the local systems ${}^p\mcl_{(\tau,d+e)}$ can be also seen as follows:
  ${}^p\QQ^H_T$ is equivariant, and hence so is
  $\TXmap_*{}^p\QQ^H_T$. Since all orbit stabilizers are connected,
  \cite[Theorem 11.6.1]{Hotta} shows that each orbit can only support
      one equivariant local system, the constant one. See
      \cite{LW-equiDcat} for more details on equivariant
      $D$-modules. 
\end{rem}
\begin{ex}\label{ex:d=4}
We give an explicit description of the vector spaces
$L^k_{(\tau,d+e)}$ from Proposition \ref{prop:ICcomp} in the
case $d=4$ for $e\geq 1$.
Here, the $(k,e)$-entry for $L^k_{(\tau_i,d+e)}$ is a sum over all
$\gamma_j$ that arise. For example, $L^0_{(\tau_2,6)}$ is the sum over
all $\gamma_3$ of dimension $3$ with $\tau_2\subseteq
\gamma_3\subseteq \sigma$ of the terms listed under $k=0$, $e=2$ in
Table \ref{table-L0tau2}.
\\
\begin{center}
The Hodge-structures $L^k_{(\tau_0,d+e)}$ for the unique $\tau_0 \subseteq \sigma$ with $\dim \tau_0 = 0$:
\end{center}
\scalebox{0.7}{\begin{minipage}{1.4\textwidth}
   \begin{equation}
    \begin{array}{c|cccc}
      k=3&L^3_{(\tau_0,5)}=\IH^8_c(X_{\sigma / \tau_0})&0&0&0\\
      k=2&0&L^2_{(\tau_0,6)} = \frac{L^0_{(\gamma_3,5)}  \otimes \IH^6_c(X_{\gamma_3 / \tau_0})}{L^3_{(\tau_0,5)}}&0&0\\
      k=1&L^1_{(\tau_0,5)}=\IH^6_c(X_{\sigma / \tau_0})&0&L^1_{(\tau_0,7)} = \frac{L^0_{(\gamma_2,6)} \otimes \IH^4_c(\gamma_2 / \tau_0)}{L^2_{(\tau_0,6)}}&0\\
      k=0&0&L^0_{(\tau_0,6)} = \frac{L^0_{(\gamma_3,5)}  \otimes \IH^4_c(X_{\gamma_3 / \tau_0}) \oplus L^0_{(\gamma_1,5)} \otimes \IH^2_c(X_{\gamma_1 / \tau_0})}{L^1_{(\tau_0,5)}}&0&L^0_{(\tau_0,8)} = \frac{L^0_{(\gamma_1,7)} \otimes \IH^2_c(X_{\gamma_1 / \tau_0})}{L^1_{(\tau_0,7)}}\\\hline
      &e=1 &e=2 &e=3 &e=4
    \end{array}
    \end{equation}
\end{minipage}}\\[3em]
\begin{center}
The Hodge-structures $L^k_{(\tau_1,d+e)}$ for all $\tau_1 \subseteq \sigma$ with $\dim \tau_1 = 1$:
\end{center}
\scalebox{0.7}{\begin{minipage}{1.4\textwidth}
   \begin{equation}
    \begin{array}{c|cccc}
      k=3&0&0&0&0\\
      k=2&L^2_{(\tau_1,5)}=\IH^6_c(X_{\sigma / \tau_1})&0&0&0\\
      k=1&0&L^1_{(\tau_1,6)} = \frac{L^0_{(\gamma_3,5)}  \otimes \IH^4_c(X_{\gamma_3 / \tau_1})  }{L^2_{(\tau_1,5)}}&0&0\\
      k=0&L^0_{(\tau_1,5)}=\IH^4_c(X_{\sigma / \tau_1})&0&L^0_{(\tau_1,7)} = \frac{L^0_{(\gamma_2,6)} \otimes \IH^2_c(\gamma_2 / \tau_1)}{L^1_{(\tau_1,6)}}&0\\\hline
      &e=1 &e=2 &e=3 &e=4
    \end{array}
    \end{equation}
\end{minipage}}\\[3em]
\begin{center}
The Hodge-structures $L^k_{(\tau_2,d+e)}$ for all $\tau_2 \subseteq \sigma$ with $\dim \tau_2 = 2$:
\end{center}
\scalebox{0.7}{\begin{minipage}{1.4\textwidth}
   \begin{equation}\label{table-L0tau2}
    \begin{array}{c|cccc}
      k=3&0&0&0&0\\
      k=2&0&0&0&0\\
      k=1&L^1_{(\tau_2,5)} = \IH^4_c(X_{\sigma / \tau_2})&0&0&0\\
      k=0&0&L^0_{(\tau_2,6)} = \frac{L^0_{(\gamma_3,5)}  \otimes \IH^2_c(X_{\gamma_3 / \tau_2})  }{L^1_{(\tau_2,5)}} &0&0\\\hline
      &e=1 &e=2 &e=3 &e=4
    \end{array}
    \end{equation}
\end{minipage}}\\[3em]
\begin{center}
The Hodge-structures $L^k_{(\tau_3,d+e)}$ for all $\tau_3 \subseteq \sigma$ with  $\dim \tau_3 = 3$:
\end{center}
\scalebox{0.7}{\begin{minipage}{1.4\textwidth}
   \begin{equation}
    \begin{array}{c|cccc}
      k=3&0&0&0&0\\
      k=2&0&0&0&0\\
      k=1&0&0&0&0\\
      k=0&L^0_{(\tau_3,5)} = \IH^2_c(X_{\sigma / \tau_3})&0&0&0\\\hline
      &e=1 &e=2 &e=3 &e=4
    \end{array}
    \end{equation}
\end{minipage}}\\[3em]

The table for $\sigma = \tau_4$ is determined by Proposition
\ref{prop:ICcomp}, Properties \eqref{ICcomp-2} and
\eqref{ICcomp-3}; it has only zero entries since $\sigma$ only
contributes to weight $d$.
\schluss\end{ex}

\subsection{An explicit formula}

If we set
\[
\ih^k_c(X_{\gamma / \tau}) := \dim_\mbq \IH^k_c(X_{\gamma /
  \tau})
\]
we can rewrite the dimension of $L^0_{(\gamma,k)}$ in 
Example \ref{ex:d=4} as follows:
\tiny
\begin{align*}
\dim_\mbq L^0_{(\tau_3,5)} =&\; \ih^2_c(X_{\sigma/\tau_3})\\
\dim_\mbq L^0_{(\tau_2,6)} =&\; \ih^2_c(X_{\sigma/\gamma_3})
\ih^2_c(X_{\gamma_3 /\tau_2}) - \ih^4_c(X_{\sigma /\tau_2}) \\
\dim_\mbq L^0_{(\tau_1,7)} =&\; \ih^2_c(X_{\sigma/\gamma_3})
\ih^2_c(X_{\gamma_3 /\gamma_2}) \ih^2_c(X_{\gamma_2/\tau_1}) -
\left[\ih^4_c(X_{\sigma /\gamma_2}) \ih^2_c(X_{\gamma_2/\tau_1})
+\ih^2_c(X_{\sigma / \gamma_3}) \ih^4_c(X_{\gamma_3 / \tau_1})\right] +
\ih^6_c(X_{\sigma/\tau_1}) \\
\dim_\mbq L^0_{(\tau_0,8)} =&\; \ih^2_c(X_{\sigma/\gamma_3})
\ih^2_c(X_{\gamma_3 /\gamma_2}) \ih^2_c(X_{\gamma_2/\gamma_1})
\ih^2_c(X_{\gamma_1 /\tau_0})\\
&\phantom{=}- \left[\ih^4_c(X_{\sigma /\gamma_2})
\ih^2_c(X_{\gamma_2/\gamma_1})\ih^2_c(X_{\gamma_1 /\tau_0})
+\ih^2_c(X_{\sigma / \gamma_3}) \ih^4_c(X_{\gamma_3 / \gamma_1})\ih^2_c(X_{\gamma_1 /\tau_0})
+\ih^2_c(X_{\sigma/\gamma_3})  \ih^2_c(X_{\gamma_3
  /\gamma_2})\ih^4_c(X_{\gamma_2 /\tau_0})\right]\\
&\phantom{=}+\left[\ih^6_c(X_{\sigma/\gamma_1})\ih^2_c(X_{\gamma_1 /\tau_0}) +
\ih^4_c(X_{\sigma /\gamma_2})\ih^4_c(X_{\gamma_2 /\tau_0})
+\ih^2_c(X_{\sigma /\gamma_3})\ih^6_c(X_{\gamma_3 / \tau_0})\right] -
\ih^8_c(X_{\sigma / \tau_0})\\
\dim_\mbq L^0_{(\tau_1,5)} =&\; \ih^4_c(X_{\sigma /\tau_1}) \\
\dim_\mbq L^0_{(\tau_0,6)} =&\; \left[\ih^2_c(X_{\sigma/\gamma_3})\ih^4_c(X_{\gamma_3 / \tau_0}) + \ih^4_c(X_{\sigma / \gamma_1})\ih^2_c(X_{\gamma_1 / \tau_0})\right] - \ih^6_c(X_{\sigma / \tau_0})
\end{align*}
\normalsize
Again, each expression is to be summed over all possible faces
$\gamma_i$ of dimension $i$ that satisfy the requisite containment
conditions. 

The particular structure of the formulas for the dimension of these
local systems is not coincidental. Our
next task is to turn recursion \eqref{eqn-L-recursion} for
$L^0_{(\gamma,k)}$ into a general explicit combinatorial formula.

We set
\[
\mu_\tau^\sigma(e):=\dim_\QQ (L^0_{(\tau, d+e)})
%
\]
for the rank of the constant local system ${^p}\mcl_{(\tau,d+e)}$ corresponding
to the intersection complex $\IC_{X_\tau}({^p}\mcl_{(\tau,d+e)})$
occurring in $\gr^W_{d+e}\TXmap_*{^p}\QQ^H_T$. 
We further introduce the following abbreviations.
\begin{ntn}
  Let
  \begin{align}
    \ih^\gamma_\tau(k)&:=\dim_\QQ(\IH_c^{d_\gamma-d_\tau+k}(X_{\gamma/\tau}));\\
    \ell_\gamma(k,e)&:=\dim_\QQ(L^k_{(\gamma,d+e)}).
  \end{align}
  Then
  \[
  \ell_\tau(k,1)=\dim_\QQ(L^k_{\tau,d+1})=
  \dim_\QQ(\IH^{d_\sigma-d_\tau+k+1}_c(X_{\sigma/\tau}))=\ih^\sigma_\tau(k+1)
  \]
  by Proposition \ref{prop:ICcomp}.\eqref{ICcomp-4}, while the recursion
  \eqref{eqn-L-recursion} yields
  \begin{equation}\label{eqn-L-recursion2}
  \ell_\tau(k,e)=\left(\sum_{\gamma\supseteq
    \tau}\ell_\gamma(0,e-1)\cdot \ih^\gamma_\tau(k+1)\right)-\ell_\tau(k+1,e-1).
  \end{equation}

  Let $0<t\in\NN$ and let $\pi=[\pi_1,\ldots,\pi_m]\partitions t$ be a
  \emph{partition}.
   of $t$ of \emph{length}
  $|\pi|=m$.  (We always assume that ``partition'' implies that each
  $\pi_j$ is nonzero, and that the entries are ordered. The partitions
  of $3$ are $[3]$, $[1,2]$, $[2,1]$ and $[1,1,1]$).  We consider
  \emph{flags} $\Gamma=
  (\gamma_{d_0}\subsetneq\gamma_{d_1}\subsetneq\ldots\subsetneq\gamma_{d_m})$
  of faces of $\sigma$, of length $|\Gamma|=m$. Here, $d_i$ is the
  dimension of $\gamma_{d_i}$. Denote by $\ih_\Gamma(\pi)$ the product
  \[
  \ih_\Gamma(\pi):= \ih_{\gamma_{d_0}}^{\gamma_{d_1}}(\pi_1)\cdot\ldots\cdot
  \ih_{\gamma_{d_{m-1}}}^{\gamma_{d_m}}(\pi_m).
  \]
  For comparable faces $\gamma\subsetneq \gamma'$, set
  \[
  \ih^{\gamma'}_\gamma(\pi)=\sum_{\substack{|\Gamma|=|\pi|\\\Gamma=(\gamma,\ldots,\gamma')}}
  \ih_\Gamma(\pi),
  \]
  and
  \[
  \ih^{\gamma'}_\gamma(t,m)=\sum_{\substack{\pi\partitions
      t\\|\pi|=m}}\ih^{\gamma'}_\gamma(\pi).
  \]
\schluss\end{ntn}

\begin{prop}\label{prop-chains}
  The rank of the local system
  ${^p}\mcl_{(\tau,d+e)}=L^0_{(\tau,d+e)}\otimes \QQ_\tau$
  occurring in
  $\gr^W_{d+e}\TXmap_*{^p}\QQ^H_T$ is
  \[
  \mu^\sigma_\tau(e)=\sum_m(-1)^{m+d_\sigma}\ih_\tau^\sigma(e,m).
  \]
\end{prop}
\begin{proof}
  To start, note that, for $\gamma''\supseteq\gamma$ one has the ``product rule''
  \[
  \ih_\gamma^{\gamma''}(\pi''\sqcup \pi)=\sum_{\gamma''\supsetneq \gamma'\supsetneq
    \gamma} \ih^{\gamma''}_{\gamma'}(\pi'')\cdot
  \ih^{\gamma'}_\gamma(\pi)
  \]
  for any two partitions $\pi'' \partitions t''$ and $\pi \partitions t$ and their juxtaposition
  $\pi''\sqcup\pi =[\pi''_1,\ldots,\pi''_{m''},\pi_1,\ldots,\pi_{m}]\partitions (t''+ t)$.  

  For each $\tau$, place the numbers $\ell_\tau(k,e)$ on a grid of
  integer points in the first quadrant of a page associated to $\tau$
  as follows:

\begin{equation} \label{fig-terms1}
   ((\tau)):\qquad \begin{array}{c|cccc}
      \vdots&\vdots&\vdots&\\
      k=2&\ih^\sigma_\tau(3)=\ell_\tau(2,1)&\ell_\tau(2,2)&\ell_\tau(2,3)&\cdots\\
      k=1&\ih^\sigma_\tau(2)=\ell_\tau(1,1)&\ell_\tau(1,2)&\ell_\tau(1,3)&\cdots\\
      k=0&\ih^\sigma_\tau(1)=\ell_\tau(0,1)&\ell_\tau(0,2)&\ell_\tau(0,3)&\cdots\\\hline
      &e=1                           &e=2          &e=3          &\cdots
    \end{array}\qquad
\end{equation}

  The column $e=1$ of the $\tau$-page consists of the numbers
  $\dim_\QQ
  \IH_c^{d_\sigma-d_\tau+k+1}(X_{\sigma/\tau})=\ih^\sigma_\tau(k+1)=
  \ell_\tau(k,1)$.  Then \eqref{eqn-L-recursion2} implies that for
  $e>1$ the
  entry in row $k$ and column $e$ of page $((\tau))$ is the difference a)
  $-$ b) where
\begin{enumerate}
 \item[a)] is the sum over all $\sigma\supsetneq \gamma \supsetneq
   \tau$ of all products of $\ih^\gamma_\tau(k+1)$ with the entry in
   row $0$ and column $e-1$ on the $\gamma$-page;
 \item[b)] the entry in row $k+1$ and column $e-1$ on
  the $\tau$-page.
\end{enumerate}

Progressing along increasing column index, all entries on
each page can be rewritten as sums of products $\ih_\Gamma(\pi)$ of
intersection cohomology dimensions $\ih^{\gamma''}_{\gamma'}(t)$. We
call such product $\ih_\Gamma(\pi)$ a ``term''. It is immediate that
each term on page $((\tau))$ arises from a flag that links $\tau$ to $\sigma$
(\emph{i.e.}, $\tau=\gamma_0, \sigma=\gamma_{|\Gamma|}$) with
$|\Gamma|=|\pi|$.

 The sum in a) contains only terms $\ih_\Gamma(\pi)$ where the
  initial element of $\pi$ equals $k+1$. On the other hand, it
  follows from induction on $k$ that the terms in b) all have the
  initial element of the corresponding $\pi$ greater than $k+1$. So,
  formal cancellation of terms cannot occur in the recursion. 
  
   When a term on the $\tau$-page arises through case a) then the
   length of the term is greater (by one) than the length of the term
   on the $\gamma$-page that gave rise to it. However, that is not the
   case if it arises from case b) when it simply copied from the
   appropriate entry on the $\tau$ page, and so term length changes if
   and only if no new factor of $-1$ is acquired. In particular, the
   sign of a term is a function of the length of the term, modulo
   two. The recursion forces the term to the partition
   $[1,1,\ldots,1]$ of length $d_\sigma$ to be positive. Hence all
   terms $\ih_\Gamma(\pi)$ on each page carry a sign of
   $(-1)^{|\pi|+d_\sigma}$.
  
   Note that in a) one could allow $\gamma=\tau$ since
   $\ih_\tau^\tau(k+1)=0$. Similarly, one can admit $\gamma=\sigma$
   since $\ell_\sigma(0,e-1)=0$ for $e>1$.  The sum in a) involves
   always all possible choices of $\gamma$, $\sigma\supseteq \gamma
   \supseteq \tau$. 
   Thus, if a partition $\pi$ occurs at all in an entry on page $((\tau))$
   then $\ih_\Gamma(\pi)$ will occur in that entry for all flags
   $\Gamma$ with $|\Gamma|=|\pi|$ that start at $\tau$ and end with
   $\sigma$. In the following table we tabulate for small $k,e$ the
   partitions that occur in Figure \eqref{fig-terms1}; here, in each
   term one should sum over all $\Gamma$ of the appropriate length
   that interpolate from $\tau$ to $\sigma$ (we will write
   $\ih([1,1,2])$ instead of $\ih^\sigma_\tau([1,1,2])$ etc.\ for ease
   of readability).  \small
    \begin{equation}\label{table-ih}
    \begin{array}{c|cccc}
      \vdots&\vdots&\vdots&\\
      k=3&\ih([4])&\ih([1,4])-\ih([5])&\ih([1,1,4])-\ih([2,4])-(\ih([1,5])-
      \ih([6]))&\cdots\\
      k=2&\ih([3])&\ih([1,3])-\ih([4])&\ih([1,1,3])-\ih([2,3])-(\ih([1,4])-
      \ih([5]))&\cdots\\
      k=1&\ih([2])&\ih([1,2])-\ih([3])&\ih([1,1,2])-\ih([2,2])-(\ih([1,3])-
      \ih([4]))&\cdots\\
      k=0&\ih([1])&\ih([1,1])-\ih([2])&\ih([1,1,1])-\ih([2,1])-(\ih([1,2])-
      \ih([3]))&\cdots\\\hline
      &e=1 &e=2 &e=3 &\cdots
    \end{array}
    \end{equation}
  \normalsize 
  
   It is therefore sufficient to investigate which partitions
  occur in the $(k,e)$-entry on page $\tau$. Since the entries in
  column $e=1$ come from a unique partition, the entries in column $e$
  will come from no more than $2^{e-1}$ partitions (the variation over
  all $\gamma$ in the recursion does not affect the resulting
  partition $\pi$, only the flag $\Gamma$).  The argument that no
  cancellation can occur reveals also that no fewer than, and hence
  exactly, $2^{e-1}$ partitions occur in each entry of column $e$.
  
  The partitions $\pi$ used in the entry $(k,e)$ on page $\tau$ have weight
  $\pi_1+\ldots+\pi_{|\pi|}= e+k$, again by induction on the
  column index. But the number of ordered integer partitions  of weight $e$ with
  positive entries is exactly $2^{e-1}$. Thus, all $2^{e-1}$
  partitions of weight $e$
  actually occur in the entry $(0,e)$, and 
  \begin{itemize}
  \item for each partition, each possible flag interpolating from
    $\tau$ to $\sigma$ contributes, and no other;
  \item the term $\ih_\Gamma(\pi)$ has sign $(-1)^{|\pi|+d_\sigma}$;
  \end{itemize}
  as stated in the proposition.
\end{proof}
The recursion as evidenced in Table \eqref{table-ih} leads immediately
to the following result.
\begin{cor}
  The number of copies $\ell_\tau(k,e)$ of $\QQ_\tau$ in
  $\calH^ki^!_{\tau,\sigma}(\TXmap_*\QQ_\sigma/W_{d+e-1}\TXmap_*\QQ_\sigma)=
  \bigoplus_{\ell_\tau(k,e)}\QQ_\tau $ equals $\sum_m
  (-1)^{m+d_\sigma} \ih^\sigma_\tau(e,m)_\boldk$ where the subscript
  $\boldk$ means each partition $\pi=[\pi_1,\ldots,\pi_m]\partitions
  t$ that contributes to $\mu^\sigma_\tau(e)=\ell_\tau(0,e)$ in
  Proposition \ref{prop-chains} is replaced by
  $[\pi_1,\ldots,\pi_{m-1},\pi_m+k]\partitions (t+k)$.  \qed
\end{cor}

\subsection{Dual polytopes}

Our final step in this section is to give a compact value to the
formula in Proposition \ref{prop-chains}.  In order to carry out this
discussion we have to introduce some
notions from toric geometry.

\begin{ntn}
    Let $\tau \subseteq \gamma \subseteq \sigma$ be faces of
    $\sigma$. The \emph{quotient face} of $\gamma$ by $\tau$ is defined
    as:
\begin{eqnarray}\label{***}
\gamma /\tau &:=& (\gamma + \tau_\mbr) / \tau_\mbr \subseteq \mbr^d / \tau_\mbr.
\end{eqnarray}

We define the dual cone and the annihilator of $\gamma$ by 
\[
\gamma^\vee := \{ y \in (\mbr^d)^\dual \mid y(x) \geq 0 \; \forall x
\in \gamma \}\qquad \text{and} \qquad \gamma^\perp := \{ y \in (\mbr^d)^\dual \mid y(x) = 0 \;
\forall x \in \gamma \}.
\]

For faces $\tau$ and $\gamma$ of $\sigma$, 
$[ \tau \subseteq \gamma \subseteq \sigma] \; \Leftrightarrow \;
[\tau^\vee \supseteq \gamma^\vee \supseteq \sigma^\vee] $ and $ [\tau
\subseteq \gamma \subseteq \sigma] \; \Leftrightarrow \; [\tau^\perp \supseteq
\gamma^\perp \supseteq \sigma^\perp] $.

There is an containment-reversing bijection 
\[
\tau \quad \longleftrightarrow \quad \tau^\star :=\tau^\perp \cap
\sigma^\vee
\]
between faces $\tau$ of
$\sigma$ of dimension $r$ and \emph{complementary faces} $\tau^\star$ of
$\sigma^\vee$ of dimension $d-r$.

The notions of dual and annihilator as well as
complementary face are relative to $\sigma$, although we usually suppress it
in the notation.
\schluss\end{ntn} 

\begin{rem}
  We record two properties of $\mu$ that will be used later.
  \begin{enumerate}
    \item The numbers $\mu^\sigma_\tau(e)$ are relative in the sense
      that they only depend on the quotient variety $X_{\sigma/\tau}$:
      Proposition \ref{prop-chains} shows that
      $\mu^\sigma_\tau(e)=\mu^{\sigma/\tau}_{\tau/\tau}(e)$.
    \item  We derive a second recursive formula. Indeed,
  as an alternating sum of weight $e$ over all flags interpolating
  from $0$ to $\sigma$, sorting the terms by their first non-trivial flag
  entry $\gamma$, one obtains
  \begin{gather}\label{eqn-recursion-alt}
    \mu^\sigma_0(e)=(-1)^{d_\sigma+1}\ih^\sigma_0(e)+\sum_{0\subsetneq
      \gamma\subsetneq
      \sigma}\left((-1)^{d_\gamma-1}\sum_k\mu^\sigma_\gamma(e-k)\cdot
    \ih^\gamma_0(k)\right).
  \end{gather}
  Here, the first summand corresponds to $\pi=[e]$, the sum collects all
  others.  Moreover, the additional power of $-1$ in all terms in the
  sum is owed to the fact that all partitions contributing to
  $\mu^\sigma_\gamma(e-k)$ are one step shorter than their avatars,
  the partitions of $e$.
  \end{enumerate}
\schluss\end{rem}

Define
\begin{align*}
\gamma^\mho &:= \{y \in (\mbr^d)^\dual/\gamma^\perp \mid y(x) \geq 0\;
\forall x \in \gamma\}.
\end{align*}
Since $(\gamma_\mbr)^\dual \simeq (\mbr^d)^\dual / \gamma^\perp$
naturally, $\gamma^\mho$ is the dual of $\gamma$ in its own span,
hence absolute (independent of $\sigma$).

We have the following basic lemma on the dual of the cone $\gamma
/\tau$ relative to $\gamma_\mbr / \tau_\mbr$. 

\begin{lem}
Let $\tau \subseteq \gamma$ be faces of $\sigma$. Then
\[
\left(\gamma /\tau\right)^\mho \simeq \tau^\star/ \gamma^\star,
\]
the right hand side computed relative to $\sigma$.
\end{lem}
\begin{proof}
We have
$(\RR^d/\tau_\RR)^\star/(\gamma/\tau)^\perp_{\sigma/\tau}=\tau^\perp_\sigma/\gamma^\perp_\sigma$,
computing on the left relative to $\sigma/\tau$ and on the right
relative to $\sigma$. 
 We have thus:
\begin{align*}
\left(\gamma /\tau\right)^\mho &= \{ y \in (\mbr^d/\tau_\RR)^\star
/(\gamma/\tau)^\perp = \tau^\perp / \gamma^\perp \mid y(x) \geq 0 \;\; \forall
\; x \in \gamma / \tau\} \\ &= (\gamma^\vee  \cap
\tau^\perp) / \gamma^\perp \\ &= ((\sigma^\vee + \gamma^\perp)\cap\tau^\perp) /
\gamma^\perp \\ &= (\sigma^\vee
\cap \tau^\perp + \gamma^\perp) / \gamma^\perp \\ &\simeq
(\sigma^\vee\cap\tau^\perp)/(\sigma^\vee\cap\tau^\perp\cap\gamma^\perp)\\ &=
\tau^\star /
\gamma^\star
\end{align*}
where the third equality follows from $\gamma^\vee = \sigma^\vee +
\gamma^\perp$ (cf. the proof of Proposition 2 on \cite[p.13]{Fulton}) and at
the end we use the second 
isomorphism theorem.
\end{proof}

\begin{defn}\label{defn:Y}
  If $\tau\subseteq \gamma$ are faces of $\sigma$, denote $Y_{\gamma/\tau}$ the
  spectrum of
  the semigroup ring
  induced by the dual cone of $\sigma/\tau$ in its natural
  lattice. In other words, the cone $\gamma / \tau$ together with its
  faces defines a fan in $\gamma_\mbr / \tau_\mbr$. The corresponding
  toric variety is
\begin{equation*}
Y_{\gamma / \tau} := X_{\tau^\star / \gamma^\star} = X_{(\gamma /
  \tau)^\mho}.
\end{equation*}

\schluss\end{defn}
The following lemma compares the intersection cohomology Betti numbers
of $Y_{\sigma /\gamma}$ with those of $X_{\gamma / 0} =X_\gamma$.

\begin{lem}\label{lem:formulaStanley}
Let $\sigma$ be a strongly convex rational polyhedral cone of
dimension $d$ as always. Then
\[
\sum_{0 \subseteq \gamma \subseteq \sigma} (-1)^{d_\gamma}
\left(\sum_i \ih^{2i}(Y_{\sigma / \gamma})\, t^i \right)\left(\sum_j
\ih^{2j}(X_{ \gamma})\, t^j \right) = 0.
\]
\end{lem}
\begin{proof}

To a cone $\sigma \subseteq \RR A= \mbr^d$ belongs the affine toric variety
$X_\sigma=\Spec \mbc[ \sigma \cap \mbz^d]$. Here is an overview of the
proof. We first explain independence of $\ih^\bullet(-)$ of the lattice
used to produce $X_\sigma$. We then discuss combinatorial intersection
homology and how it applies to quotient polytopes and cones. Finally,
we put the pieces together, using results of Stanley.

Now let $ N \subseteq
\mbr^d$ be another $\mbz$-lattice (a free subgroup of rank $d$ whose
$\QQ$-span is $\QQ A$). The affine toric variety $X^N_\sigma := \Spec
\mbc[\sigma \cap N]$ can be different from $X_\sigma$, but we have
a canonical isomorphism
\begin{equation}\label{eq:IHdiffL}
\IH^\bullet(X_\sigma) \simeq \IH^\bullet(X^N_\sigma).
\end{equation}
This can be seen as follows: Consider the lattices $N'\supseteq N$ in
$\mbr^d$. It is enough to prove
that $\IH^\bullet(X^{N'}_\sigma) \simeq \IH^\bullet(X^N_\sigma)$. The finite group
$G:= N' / N$ naturally acts on $X^{N'}_\sigma$, and $X^N_{\sigma}$ is
the quotient of $X^{N'}_\sigma$ under this action
(cf. \cite[Proposition 1.3.18]{CLS}). We have the following
isomorphism
\[
\IH^\bullet(X^{N}_\sigma) \simeq \IH^\bullet(X^{N'}_\sigma)^G = \IH^\bullet(X^{N'}_\sigma)
\]
where $\IH^\bullet(X^{N'}_\sigma)^G$ is the $G$-invariant part. The
isomorphism follows from \cite[Lemma 2.12]{Ki} and the equality comes
from the fact that the action of $G$ is induced by the action of the
open dense $\CC$-torus of $X^{N'}_\sigma$ which acts trivially: a
$\CC$-torus acting continuously on a rational vector space must have a
dense subset acting trivially; continuity forces triviality
everywhere. Hence when writing $\IH^\bullet(X_\sigma)$ we do not need to
worry about the lattice with respect to which $X_\sigma$ is defined.

Assume that we are given a rational polytope $P\subseteq \mbr^{d-1}$
of dimension $d-1$. The set of faces of $P$ (including the empty face
$\emptyset$), ordered by inclusion, forms a poset. Given such a
polytope, Stanley \cite{Stanley} defined polynomials
\begin{align}
g(P) = \sum g_i(P) t^i \qquad \text{and} \qquad h(P) = \sum h_i(P) t^i
\end{align}
recursively by
\begin{itemize}
\item $g(\emptyset) = 1$;
\item $h(P) = \sum_{\emptyset \leq F < P} (t-1)^{\dim P - \dim F -1} g(F)$;
\item $g_0(P)= h_0(P)$, \qquad $g_i(P) = h_i(P) - h_{i-1}(P)$ for $0 <
  i \leq \dim P/2$ and $g_i(P) = 0$ for all other $i$.
\end{itemize}

Now assume $0$ is in the interior $\Int(P)$. From such a polytope we get a fan
$\Sigma_P$ by taking the cones over the faces of $P$; here the empty
face corresponds to the cone $\{0\} \subseteq \mbr^{d-1}$. This gives a
projective toric variety $X_P$ together with an embedding into
projective space. It was proved independently by Denef
and Loeser \cite{DF1} and Fieseler \cite{Fies} that
\[
h_i(P) = \ih^{2i}(X_P).
\]
Denote by $\cone(X_P)$ the affine cone of $X_P$. Then
\[
g_i(P) = h_i(P) - h_{i-1}(P) = \ih^{2i}(\cone(X_P)) \qquad \text{for}
\; 0< i \leq \dim (P)/2.
\]
The affine cone of $X_P$ has the following toric description: Consider
the embedding of $P\subseteq \mbr^{d-1}$ in $\mbr^d$ under the map $i:
x \mapsto (1,x)$. Let $\Cone(P)$ be the (rational, polyhedral, strongly
convex) cone over $i(P)$ with apex at the origin. Then
$\cone(X_P)$ is an affine toric variety given by $\cone(X_P)=X_{\Cone(P)^\vee} =
\Spec \mbc[\Cone(P)^\mho \cap (\mbz^d)^\dual]$. Hence we get
\begin{equation}\label{eq:giIH}
g_i(P) = \ih^{2i}(X_{\Cone(P)^\mho})\, .
\end{equation}

Two polytopes $P_1$ and $P_2$ are \emph{combinatorially equivalent} if
they have isomorphic face posets, denoted $P_1 \sim P_2$. This is an
equivalence relation, and $g(P)$ and $h(P)$ only depend on the
equivalence class $[P]$ of $P$. Similarly, given two strongly convex
rational polyhedral cones $\sigma_1$ and $\sigma_2$ we write $\sigma_1
\sim \sigma_2$ if their face posets are isomorphic. If we have
$\sigma_i = \Cone(P_i)$ for $i=1,2$ then $[P_1 \sim
  P_2]\Leftrightarrow [\Cone(P_1) \sim \Cone(P_2)]$.

For a given rational polytope $P$ with $0 \in \Int(P)$, the
dual polytope is
\[
P^\circ := \{ x\in (\RR P)^\dual \mid x(y) \geq -1 \; \forall y \in P\},
\]
$\RR P$ being the affine span of $P$.  There is an order-reversing
bijection of the $k$-dimensional faces $F$ of $P$ and the
$(\dim(P)-1-k)$-dimensional faces $\{x\in P^\circ\mid x(F)=-1\}$ of
$P^\circ$.

If the origin is not in $\Int(P)$, translate $P$ so that
$0\in\Int(P)$ and then dualize.  The combinatorial equivalence
class of the dual is then well-defined and we still
write $P^\circ$ for this class.

From a $k$-dimensional face $F$ of the $(d-1)$-dimensional polytope
$P$ we construct an equivalence class of $(d-k-2)$-dimensional
polytopes $P/F$ as follows. Choose a $(d-k-2)$-dimensional affine
subspace $L$ whose intersection with $P$ is a single point of the
interior of $F$. Then a representative of $P/F$ is given by $L'\cap P$
where $L'$ is another $(d-k-2)$-dimensional affine subspace, near $L$
in the appropriate Grassmannians, and such that it meets an interior
point of $P$.  (One checks that this representative is well-defined up
to projective transformation, hence the combinatorial type is
well-defined).  One can see easily that the cone over $P/F$ is exactly
$\Cone(P)/\Cone(F)$, compare \eqref{***}:
\begin{equation}\label{eq:PmodF}
\Cone(P/F) \sim \Cone(P)/\Cone(F)=(\Cone(P)+\RR F)/\RR F.
\end{equation}

We will prove Lemma \ref{lem:formulaStanley} using the following
formula by Stanley \cite{Stanley2} (we use here a presentation given
by Braden and MacPherson in \cite[Proposition 8, formula (3)]{BraMac}):
\begin{equation}\label{eq:formulaStan}
\sum_{\emptyset \subseteq F \subseteq P} (-1)^{\dim F} g(F^\circ)g(P/F) = 0
\end{equation}

The dual $F^\circ$ of a rational polytope $F$ is rational in many
lattices. Choosing one such lattice yields a rational, polyhedral, strongly
convex cone $\Cone(F^\circ)$ for which $\Cone(F^\circ)^\mho$ is
well-defined. By \eqref{eq:IHdiffL},  its intersection homology is
independent of the lattice choice.
It follows that, with $\gamma$ the cone over $F$,
\begin{equation}\label{eq:formgi1}
g_i(F^\circ) = \ih^{2i}(X_{\Cone(F^\circ)^\mho}) \simeq \ih^{2i}(X_{\Cone(F)})
\simeq \ih^{2i}(X_\gamma)
\end{equation}
where we used formula \eqref{eq:IHdiffL} for the last
isomorphism. Recalling Definition \ref{defn:Y} and that
$\Cone(P)=\sigma$, we obtain
\begin{equation}\label{eq:formgi2}
g_i(P/F) = \ih^{2i}(X_{\Cone(P/F)^\mho}) = \ih^{2i}(X_{(\Cone(P) /
  \Cone(F))^\mho}) = \ih^{2i}(Y_{\Cone(P) / \Cone(F)}) =
\ih^{2i}(Y_{\sigma / \gamma}),
\end{equation}
where the first equality is \eqref{eq:giIH}, the second equality
follows from \eqref{eq:PmodF}, the third equality is Definition \ref{defn:Y},
and the last follows from \eqref{eq:IHdiffL}.  Plugging
\eqref{eq:formgi1} and \eqref{eq:formgi2} into \eqref{eq:formulaStan}
and multiplying with $(-1)$ we get the statement of the Lemma.
\end{proof}

We are now ready to give our main result about the weight filtration
on the inverse Fourier--Laplace transform of the $A$-hypergeometric system $H_A(0)$:

\begin{thm}\label{thm:weightonhQ}
The associated graded
module to the  weight filtration  on the mixed
Hodge module $\TVmap_*({^p}\mbq^H_T)$ is for $e=0,\ldots,d$ given by 
\[
\gr^W_{d+e} \TXmap_*({^p}\mbq^H_T) \simeq \bigoplus_\tau\IC_{X_\tau}({^p}
\mcl_{(\tau,d+e)}), 
\]
where
\[
  {^p} \mcl_{(\tau,d+e)} = L^0_{(\tau,d+e)} \otimes
  {^p}\mbq^H_{T_\tau}
  \]is a constant variation of Hodge structures of
weight $d+e$ on $T_\tau$. Here $L^0_{(\tau,d+e)}$ is a Hodge-structure of
Hodge--Tate type of weight $d+e-d_\tau$ of dimension
  \[
\mu^\sigma_\tau(e)=\dim_\mbq L^0_{(\tau,d+e)} =
\ih^{d_\sigma-d_\tau+e}_c(Y_{\sigma/\tau}),
  \]
compare Definition \ref{defn:Y}.
\end{thm}
\begin{proof}
In light of Proposition \ref{prop-chains} (and Kashiwara equivalence)
it only remains to prove that $\mu^\sigma_\tau(e)=\dim_\mbq
L^0_{(\tau,d+e)}$ equals
$\ih^{d_\sigma-d_\gamma+e}_c(Y_{\sigma/\tau})$.

  An inspection shows that if $\sigma=\tau$ then
  the theorem is (trivially) correct. We argue by induction on
  $d_\sigma-d_\tau$. While in principle a Poincar\'e series only
  involves non-negative terms there is no harm in allowing negative
  indices: they just add zero terms.
  
  According to Lemma \ref{lem:formulaStanley} we have
\makeatletter
\tagsleft@false

  \begin{align*}
    0&=\sum_{0\subseteq \gamma\subseteq \sigma}(-1)^{d_\gamma}\left(
    \sum_{j=-\infty}^\infty t^j\cdot\ih^{2j}(Y_{\sigma
      /\gamma})\right)\cdot \left( \sum_{i=-\infty}^\infty
    t^i\cdot\ih^{2i}(X_\gamma)\right)\\
    &=\sum_{0\subseteq \gamma\subseteq \sigma}(-1)^{d_\gamma}\left(
    \sum_{j=-\infty}^\infty t^j\cdot
    \ih^{2(d_\sigma-d_\gamma-j)}_c(Y_{\sigma/\gamma})\right)\cdot
    \left( \sum_{i=-\infty}^\infty t^i\cdot
    \ih^{2(d_\gamma-i)}_c(X_\gamma)\right)\tag{Lemma \ref{lem:propIHtoric}}\\
    &=(-1)^{d_\sigma}\cdot\sum_{k=-\infty}^\infty t^k\cdot
    \ih^{2(d_\sigma-k)}_c(X_\sigma)\tag{from
      $\gamma=\sigma$}\\
    &\phantom{==}+\sum_{0\subsetneq \gamma\subsetneq \sigma}(-1)^{d_\gamma}\left(
    \sum_{j=-\infty}^\infty
    t^j\cdot\ih^{2(d_\sigma-d_\gamma-j)}_c(Y_{\sigma
      /\gamma})\right)\cdot \left( \sum_{j=-\infty}^\infty
    t^i\cdot\ih^{2(d_\gamma-i)}_c(X_\gamma)\right)\phantom{xxxxxxxxxx}\tag{general $\gamma$}\\
    &\phantom{==}+\sum_{k=-\infty}^\infty t^k\cdot
    \ih^{2(d_\sigma-k)}_c(Y_{\sigma /0}). \tag{from
      $\gamma=0$}
  \end{align*}
where we have used Lemma \ref{lem:propIHtoric} (1) for the second equality.

  Induction allows to substitute
  $\mu^\sigma_\gamma(d_\sigma-d_\gamma-2j)$ for
  $\ih^{2(d_\sigma-d_\gamma-j)}_c(Y_{\sigma/\gamma})$ for all
  $\gamma\not=0,\sigma$ in the sum ``general $\gamma$''.  At the same
  time we can replace, by definition,
  $\ih^{2(d_\gamma-i)}_c(X_\gamma)$ by $\ih^\gamma_0(d_\gamma-2i)$.
  With these substitutions, collect terms with equal $t$-power:

  \begin{align*}
    0&=(-1)^{d_\sigma}\cdot\sum_{k=-\infty}^\infty t^k\cdot
    \ih^\sigma_0(d_\sigma-2k)\tag{from $\gamma=\sigma$}\\
    &\phantom{==}+\sum_{k=-\infty}^\infty \left(\sum_{i+j=k}t^k\sum_{0\subsetneq
      \gamma\subsetneq \sigma} (-1)^{d_\gamma}\left( 
    \mu^\sigma_\gamma(d_\sigma-d_\gamma-2j)\cdot
    \ih^\gamma_0(d_\gamma-2i)\right)\right)\phantom{xxxx}\tag{general $\gamma$}\\ 
    &\phantom{==}+\sum_{k=-\infty}^\infty t^k\cdot \ih^{2(d_\sigma-k)}_c(Y_{\sigma/0})
    \tag{from $\gamma=0$}.
  \end{align*}
\tagsleft@true
In degree $k$ we have therefore:
  \begin{gather}\label{eqn-mu-Y}
   0= (-1)^{d_\sigma}\cdot
   \ih^\sigma_0(d_\sigma-2k)+\sum_{i+j=k}\left(\sum_{0\subsetneq
      \gamma\subsetneq \sigma} (-1)^{d_\gamma}
    \mu^\sigma_\gamma(d_\sigma-d_\gamma-2j)\cdot
    \ih^\gamma_0(d_\gamma-2i)\right)
    + \ih^{2(d_\sigma-k)}_c(Y_{\sigma/0}).
  \end{gather}
Since the odd-dimensional intersection homology Betti numbers are zero (cf. Lemma \ref{lem:propIHtoric} (3)),
we can include all missing summands $(-1)^{d_\gamma}
\mu^\sigma_\gamma(d_\sigma-d_\gamma-j')\cdot
\ih^\gamma_0(d_\gamma-i')$ with $i'+j'=2k$ without affecting the value
of the sum. Since
$\ih^\gamma_0(d_\gamma-i')=\ih_c^{d_\gamma+(d_\gamma-i')}(X_\gamma)$ , no summand
with $i'':=d_\gamma-i'\le 0$ can contribute  (cf. Lemma \ref{lem:propIHtoric} (2)).
We can therefore rewrite \eqref{eqn-mu-Y} to
  \begin{gather}\label{eqn-mu-Y2}
   0= (-1)^{d_\sigma}\cdot
   \ih^\sigma_0(d_\sigma-2k)+\sum_{i''}\left(\sum_{0\subsetneq
      \gamma\subsetneq \sigma} (-1)^{d_\gamma}
    \mu^\sigma_\gamma(d_\sigma-2k-i'')\cdot
    \ih^\gamma_0(i'')\right)
    + \ih^{2(d_\sigma-k)}_c(Y_{\sigma/0}).
  \end{gather}

  In light of the
  recursion \eqref{eqn-recursion-alt}, this yields
  $0=-\mu^\sigma_0(d_\sigma-2k)+\ih^{2(d_\sigma-k)}_c(Y_{\sigma/0})$ and
  finishes the inductive step.
\end{proof}

\section{Weight filtrations on $A$-hypergeometric systems}

In this section we translate the results from the previous section to
hypergeometric $D$-modules on
\[
V:=\CC^n
\]
via the Fourier transform. Part of this is rather
mechanical, but identifying the weight filtrations requires some
extra hypotheses, see Corollary \ref{cor:monodromic}.

\subsection{Translation of the filtration}

We start this section with various definitions around 
$A$-hypergeometric systems. For more details, we refer to (for
example) \cite{MMW05,RSW18}. Our terminology is that of \cite{MMW05}.

Throughout, we continue Notation \ref{ntn-A}
\begin{defn}
Write $\mbl_A$ for the $\mbz$-module of integer relations among the
columns of $A$ and write $\mcd_{\mbc^n}$ for the sheaf of rings of
differential operators on $V=\mbc^n$ with coordinates $x_1,\ldots,
x_n$. Denote $\del_j$ the operator $\del/\del x_j$.  For $\beta=(\beta_1,\ldots,\beta_d)\in\CC^d$ define
\[
\mcm^\beta_A := \mcd_{\mbc^n} / \mci_A^\beta
\]
where $\mci_A^\beta$ is the sheaf of left ideals generated by the
\emph{toric operators}
\[
\Box_{\boldu} := \prod_{u_j < 0} 
\p_{j}^{-u_j} - \prod_{u_j > 0} \p_{j}^{u_j}
\]
for all $\boldu=(u_1,\ldots, u_n) \in \mbl_A$, and the \emph{Euler
  operators}
\[
E_i:= \sum_{ j=1}^n a_{ij} x_j \p_{j} - \beta_i.
\]
\schluss\end{defn}
We will write $M^\beta_A := \Gamma(V, \mcm^\beta_A)$ for the
$D_{A}$-module of global sections where
$D_{A}=\Gamma(V,\calD_{V})$. Denote by $R_A$
(resp. $O_A$) the polynomials rings over $\mbc$ generated by $\p_A =
\{\p_j\}_j$ (resp. $x_A = \{x_j\}_j$).  and set $S_A := R_A /
R_A\{\Box_{\boldu}\}_{\boldu \in \mbl_A}$.

We have 
\begin{align}
x^{\boldu} E_i - E_i x^{\mathbf{u}} &= -(A \cdot \boldu)_i x^{\boldu}, \notag  \\
\p^{\boldu} E_i - E_i \p^{\mathbf{u}} &= (A \cdot \boldu)_i
\p^{\boldu}. \notag
\end{align}

Define the $A$-degree on $R_A$ and $D_A$ as 
\[
\deg_A(x_j) = \mathbf{a}_j = -\deg_A(\p_j)\in\ZZ A
\]
and denote by $\deg_{A,i}(-)$ the degree associated to the $i$-th row
of $A$. This convention agrees with the choices in
  \cite{MMW05} but is opposite to that in \cite{Reich2}. Then $E_i P = P(E_i + \deg_{A,i}(P))$ for any $A$-graded $P \in
D_A$.

Given a left $A$-graded $D_A$-module $M$ we can define commuting
$D_A$-linear endomorphisms $E_i$ via
\[
E_i \circ m := (E_i - \deg_{A,i}(m))\cdot m
\]
for $A$-graded elements of $M$.
If $N$ is an $A$-graded $R_A$-module $N$ we get a commuting set of
$D_A$-linear endomorphisms on the left $D_A$-module $D_A \otimes_{R_A}
N$ by
\[
E_i \circ ( P \otimes Q) := (E_i - \deg_i(P) - \deg_i(Q)) P \otimes Q \, 
\]
for any $A$-graded $P,Q$.  The \emph{Euler--Koszul complex}
$K_\bullet(M;E-\beta)$ of the $A$-graded $R_A$-module $N$ is the homological
Koszul complex induced by $E-\beta := \{(E_i -\beta_i)\circ\}_i$ on
$D_A \otimes_{R_A} N$. The terminal module sits in homological degree
zero. We denote by $\mck_\bullet(N; E-\beta)$ the corresponding
complex of quasi-coherent sheaves. The homology objects are
$H_\bullet(N;E-\beta)$ and $\mch_\bullet(N;E-\beta)$, respectively.

For a finitely generated $A$-graded $R_A$-module $N = \bigoplus_\alpha
N_\alpha$ write $\deg_A(N) = \{ \alpha \in \mbz A \mid N_\alpha \neq
0\}$ and then let the \emph{quasi-degrees} of $N$ be
\begin{eqnarray*}
\qdeg_A(N)&:=& \overline{\deg_A(N)}^{Zar},
\end{eqnarray*}
the Zariski closure of $\deg_A(N)$ in $\mbc^d$.

The following subset of parameters $\beta \in \mbc^d$ will be of
importance to us.
\begin{defn}[\cite{SchulzeWalther-ekdi}]
The set of \emph{strongly resonant parameters} of $A$ is
\[
\sRes(A) := \bigcup_{j=1}^d \sRes_j(A)
\]
where
\[
\sRes_j(A) := \left\{\beta \in \mbc^d \mid \beta \in -
(\mbn+1)\mathbf{a}_j - \qdeg(S_A / (\p_j))\right\}.
\]
\schluss\end{defn}

\begin{defn}
Let
\[
\langle - , - \rangle : \overbrace{\mbc^n}^{\hatV} \times \overbrace{\mbc^n}^{V} \ra \mbc, \qquad
(\fraky_1,\ldots ,\fraky_n,\frakx_1,\ldots, \frakx_n,) \mapsto
\sum_{i=1}^n \frakx_i \fraky_i \,.
\]
We define a $\mcd_{\hatV \times V}$-module by
\[
\mcl := \mco_{\hatV\times V}\cdot  \exp({\langle -, - \rangle}),
\]
and we refer to \cite[Section 5]{KS97} for details on these
sheaves. Denote by $p_i: \hatV \times V \ra \mbc^n$ for $i=1,2$
the projection to the first and second factor respectively
(identifying the respective factor with the target). The
\emph{Fourier--Laplace} transform is defined by
\begin{eqnarray}
\FL\colon \rmD^b_{qc}(\mcd_{\hatV}) &\lra& \rmD^b_{qc}(\mcd_{V}) \notag, \\
 \mcm &\mapsto& p_{2+}(p_1^+ \mcm \overset{L}\otimes \mcl)[-n] \notag
\end{eqnarray}
with $\FL\circ\FL=-\id$.
\schluss\end{defn}
We denote by $\hat M^\beta_A$ the module of global sections
to the sheaf
\[
\hat{\mcm}^\beta_A := \FL^{-1}(\mcm^\beta_A)
\]
and define the following twisted structure sheaves on $T$:
\[
\mco_T^\beta := \mcd_T / \mcd_T \cdot (\p_t t_1 + \beta_1,\ldots, \p_{t_d} t_d+\beta_d),
\]
where we note that $\mco^\beta_T \simeq \mco^\gamma_T$ if and only if $\beta -
\gamma \in \mbz^d$.

\begin{thm}(\cite{SchulzeWalther-ekdi} Theorem 3.6, Corollary
  3.7)\label{thm-SWekdi} Let $A$ be a pointed $(d \times n)$ integer
  matrix satisfying $\mbz A = \mbz^d$. Then for the map $\TVmap$ in
  \eqref{eq:h-map}, the following statements
  are equivalent
\begin{enumerate}
\item $\beta \not \in \sRes(A)$;
\item $\hat{\mcm}^\beta_A \simeq \TVmap_{ +} \mco_T^\beta$.\qed
\end{enumerate}
\end{thm}

Theorem \ref{thm-SWekdi} implies that for
$\beta \in \mbz^d \smallsetminus \sRes(A)$ we have, with notation as
in \eqref{map-factorization},
\[
\hat{\mcm}^\beta_A \simeq \TVmap_{+} \mco_T \simeq \XVmap_+ \TXmap_+ \mco_T.
\]

We now concentrate on $\beta\in\ZZ^d\minus \sRes(A)$. Since $\calO_T$ is the underlying left
$\mcd_T$-module of ${^p}\mbq^H_T$ this induces the structure of a
mixed Hodge module on $\hat{\mcm}_A^\beta$ from Theorem
\ref{thm:weightonhQ}.
Recalling Definition
\ref{def-Dmod} and bearing in mind that
the functor $\XVmap_*$ preserves weight, we infer: 
\begin{cor}\label{cor-Mhat}
For $\beta \in \mbz^d \smallsetminus \sRes (A)$, the
module $\hat{\mcm}^\beta_A=\FL^{-1}(\mcm^\beta_A)$ carries the
structure of a mixed Hodge module ${^H\!\!}\hat{\mcm}^\beta_A$ which
is induced by the isomorphism
\[
\hat{\mcm}^\beta_A \simeq \Dmod (\XVmap_* \TXmap_* {^p}\mbq^H_T).
\]
The corresponding weight filtration is given by
\[
\gr^W_{d+e} {^H\!\!}\hat{\mcm}^\beta_A \simeq \bigoplus_\gamma
\bar \XVmap_{\gamma*}\IC_{X_\gamma}({^p} \mcl_{(\gamma,d+e)}) 
\]
where $\bar \XVmap_\gamma: X_\gamma \ra \mbc^n$ is the embedding of the
closure of the $\gamma$-torus, and ${^p} \mcl_{(\gamma,d+e)}
= L^0_{(\gamma,d+e)} \otimes {^p}\mbq^H_{T_\gamma}$ is a constant
variation of Hodge structures of weight $d+e$. Here
$L^0_{(\gamma,d+e)}$ is a Hodge-structure of Hodge--Tate type of
weight $d+e-d_\gamma$ of dimension
  \[
\dim_\mbq L^0_{(\gamma,d+e)} =
\ih^{d_\sigma-d_\gamma+e}_c(Y_{\sigma/\gamma}),
\]
with $Y_{\sigma/\gamma}$ as in Definition \ref{defn:Y}.\qed
\end{cor}

%
As a corollary, we obtain
information about the holonomic length of 
$M^0_A$. Recall that $\mcm^{\IC}(X_\gamma)=\Dmod(\IC_{X_\gamma})$ is the
unique simple $T$-equivariant $\mcd$-Module on $\hatV$ with support  $X_\gamma$.
\begin{cor}\label{cor:hatW}
Let $A$ be as in Notation \ref{ntn-A} and choose $\beta \in \mbz^d
\smallsetminus \sRes (A)$. Then $\mcm^\beta_A$ carries a finite
separated exhaustive filtration
$\{\hat{W}_\bullet \mcm^\beta_A\}_{e=0}^d$ given by
\[
\hat{W}_\bullet \mcm^\beta_A:=\FL(W_\bullet {}^H\!\!\hat\mcm^\beta_A).
\]
This filtration satisfies
\[
\gr^{\hat{W}}_{d+e} \mcm^\beta_A = \bigoplus_\gamma \bigoplus_{i=1}^{\mu^\sigma_\gamma(e)} C_{\gamma}.
\]
Here, $C_{\gamma}=\FL
\mcm^{\IC}(X_\gamma)$
 is a simple equivariant holonomic $\mcd$-module
(that is independent of $e$ and) which occurs in
$\gr^{\hat{W}}_{d+e} \mcm^\beta_A$ with multiplicity
$\mu^\sigma_\gamma(e) = \ih_c^{d_\sigma - d_\gamma
  +e}(Y_{\sigma/\gamma}) = \ih^{d_\sigma - d_\gamma
  -e}(Y_{\sigma/\gamma})$,
the  $(d_\sigma - d_\gamma -e)$-th intersection cohomology
Betti number of the affine toric variety $Y_{\sigma / \gamma}$. \qed
\end{cor}

\subsection{The homogeneous case: monodromic Fourier--Laplace}

Although the Fourier--Laplace transformation does not preserve regular
holonomicity in general, and so $\mcm^\beta_A$ may not be a mixed Hodge module, it is preserved for the derived category of complexes of $\mcd$-modules with so-called \emph{monodromic cohomology}. In this case we can express the Fourier--Laplace transformation as a  monodromic Fourier transformation (or Fourier--Sato transformation). In order to make this work, we now assume that the matrix $A$ is \emph{homogeneous}, which means that
\[
(1,\ldots, 1)^T\in \ZZ(A^T).
\]
Via a suitable coordinate change on the torus $T$, we can then assume that the top row of $A$ is $(1,\ldots,1)$.

\medskip

Denote by
\[
\theta: \mbc^\ast \times \hatV \ra \hatV
\]
the standard $\mbc^\ast$ action on $\hat V$; let $z$ be a coordinate on
$\mbc^\ast$. We refer to the push-forward $\theta_*(z \partial_z)$ as
the \emph{Euler vector field} $\mathfrak{E}$.

\begin{defn}\cite{Brylinski}
A regular holonomic $\mcd_{\hatV}$-module $\mcm$ is called \emph{monodromic}, if the Euler field $\mathfrak{E}$ acts finitely on the global sections of $\mcm$: for each global section section $v$ of $\mcm$ the set $\{\mathfrak{E}^n(v)\}_{n\in\NN}$ should generate a finite-dimensional vector space.
We denote by $D^b_{mon}(\mcd_{\hatV})$ the derived category of bounded complexes of $\mcd_{V'}$-modules with regular holonomic and monodromic cohomology. 
\schluss\end{defn}

Since we assume that $A$ has $(1,\ldots,1)$  as its top row, each $\hat{\mcm}^\beta_A$ is monodromic.

\begin{thm}\cite{Brylinski}
\begin{enumerate}
\item $\FL$ preserves complexes with monodromic cohomology.
\item In $D^b_{mon}(\mcd_{V})$ and $D^b_{mon}(\mcd_{\hatV})$ we
  have
\[
\FL \circ \FL \simeq \id \quad \text{and} \quad \mbd \circ \FL \simeq \FL \circ \mbd \, .
\]
\item $\FL$ is $t$-exact with respect to the natural $t$-structures on $D^b_{mon}(\mcd_{V'})$ resp. $D^b_{mon}(\mcd_{V})$.
\end{enumerate}
\end{thm}
\begin{proof}
The above statements are stated in \cite{Brylinski} for constructible
monodromic complexes. One has to use the Riemann-Hilbert
correspondence, \cite[Proposition 7.12, Theorem 7.24]{Brylinski} to
translate the statements. So the first statement is Corollaire 6.12,
the second statement is Proposition 6.13 and the third is Corollaire
7.23 in \cite{Brylinski}.
\end{proof}

We will now consider the monodromic Fourier--Laplace transform (or Fourier--Sato transform) which preserves the category of mixed Hodge modules.
\begin{defn}\label{defn:FStrafo}
Consider the diagram
\[
\xymatrix{
   &{\overbrace{\mbc^n \times \mbc^n}^{\hatV\times V}} \ar[dl]_{p_1}\ar[dr]^{\omega}&\\
  \hatV=\mbc^n & & \mbc_z \times V &\ar[l]_{i_0} {\overbrace{\{0 \} \times
    \mbc^n}^{\{0 \} \times V}
  }
}
  \]
where $p_1$ is the projection to the first factor, $i_0$ is the inclusion and the map $\omega$ is given by
\begin{align*}
\omega\colon \hatV \times V &\lra \mbc_z \times V \\
(\fraky,\frakx) &\mapsto (\frakz=\sum \frakx_i \fraky_i ,\fraky) \notag
\end{align*}
The \emph{Fourier--Sato transform} or \emph{monodromic Fourier transform}  is defined by
\begin{align*}
\rmD^b(\MHM(\hatV)) &\lra \rmD^b(\MHM(V)) \\
\mcm &\mapsto \phi_z \omega_* {^p}p_1^{!} \mcm  \simeq \phi_z \omega_! {^p}p_1^! \mcm\notag
\end{align*}
where $\phi_z$ is the nearby cycle
functor along $z=0$ and we write ${^p}f^! := f^![d_Y - d_X]$ for a map $f: X \ra Y$. 
The isomorphism follows from \cite[Proposition 10.3.18]{KS}.
\schluss\end{defn}
\begin{rem}
The original definition of the Fourier--Sato transform is different;
we use here an equivalent version (see \cite[Def.\ 3.7.8,
    Prop.\ 10.3.18]{KS}) that is well adapted to mixed Hodge modules.
  \schluss\end{rem}

For a monodromic complex the (usual) Fourier--Laplace transformation
and the monodromic Fourier transformation are the same (we use again
the equivalent version of the Fourier-Sato version from \cite{KS}):
\begin{thm}\cite[Th\'eor\`{e}me 7.24]{Brylinski}
Let $\mcm \in D^b_{mod}(\mcd_{\hatV})$ then 
\[
DR^{an} ( \FL(\mcm)) \simeq \phi_z \omega_* {^p}p_1^{!} DR^{an}(\mcm).\pushQED{\qed}\qedhere
\]
\end{thm}
It follows that the monodromic Fourier transform induces an exact functor
\[
\phi_z \omega_* {^p}p_1^! : \MHM(\hatV) \lra \MHM(V)
\]
We next identify a class of modules for which the monodromic Fourier
transform has a very simple effect on the weight filtration.
\begin{prop}
Let $\pi:\hatV \smallsetminus \{0\}  \ra \mbp(\hatV)$ be the natural projection and $j_0: \hatV \smallsetminus \{0\} \ra \hatV$ the inclusion. Let $\mcm \in \MHM(\hatV)$ such that $\mcm \simeq (j_0)_* \pi^! \mcn$ for some $\mcn \in \rmD^b \MHM(\mbp(\hatV))$. Then
\[
W_{k-n} \left(\phi_z \omega_* {^p}p_1^! \mcm \right) \simeq  \phi_z \omega_* {^p}p_1^! (W_k \mcm )
\]
\end{prop}

\begin{proof}
We first prove that the logarithm of the monodromy $N$ acts trivially on $\phi_z \omega_* p_1^! \mcm$. Define the subvarieties 
\begin{align*}
U &:= \{\sum_{i=1}^n \fraky_i \frakx_i \neq 0\} \subseteq \mbp(\hatV) \times V \\
\widetilde{U}&:= \{\sum_{i=1}^n \fraky_i \frakx_i \neq 0\} \subseteq (\hatV \smallsetminus \{0\}) \times V \\
U_1 &:= \{\sum_{i=1}^n \fraky_i \frakx_i = 1\} \subseteq (\hatV \smallsetminus \{0\}) \times V
\end{align*}
with the embeddings $j_U: U \ra \mbp(\hatV) \times V$ and  $\widetilde{j}: \widetilde{U} \ra (\hatV\smallsetminus \{0\}) \times V$. Notice that we have isomorphisms
\[
  \begin{minipage}{0.3\textwidth}\begin{eqnarray*}
      f\colon \mbc^* \times U_1 &\lra& \widetilde{U}\\
      (\frakz, \fraky,\frakx) &\mapsto& ( \frakz \cdot\fraky,\frakx)\end{eqnarray*}\notag
    \end{minipage}
  \qquad\qquad\text{and}\qquad\qquad
  \begin{minipage}{0.3\textwidth}\begin{eqnarray*}
      g\colon U_1 &\lra& U \\
      (\fraky,\frakx) &\mapsto& ((\fraky_1:\ldots : \fraky_n),\frakx).\end{eqnarray*}\notag
  \end{minipage}
\]
Consider now the following diagram
\[
\xymatrix{\hatV & \hatV \times V \ar[l]_{p_1} \ar[r]^{\omega} & \mbc_z \times V & \mbc^*_z \times V \ar[l]_j\\
(\hatV \smallsetminus \{0\}) \ar[u]^{j_0} \ar[d]_\pi & (\hatV \smallsetminus \{0\}) \times V \ar[u]^{j_0 \times \id} \ar[l]_{\overline{p}_1} \ar[d]_{\pi \times \id} \ar[ur]_{\overline{\omega}} & \widetilde{U} \ar[ur]^{\tilde{\omega}} \ar[d]^{\pi_U} \ar[l]^-{\tilde{j}} & \mbc^*_z \times U_1 \ar[l]_f^\simeq \ar[u]_{\id \times \xi} \ar[d]^{p_2}\\
\mbp(\hatV) & \mbp(\hatV) \times V \ar[l]_{\pi_1} & \ar[l]_-{j_U} U &U_1 \ar[l] _g
^\simeq}
\]
where $\xi: U_1 \subseteq (\hatV\smallsetminus\{0\}) \times V \ra V$ is the projection to the second factor and $\overline{p}_1,\pi_1$ resp. $\overline{\omega}$, $\tilde{\omega}$ are the corresponding restrictions of $p_1$ resp. $\omega$.

We have the following isomorphisms
\begin{align*}
j^! \omega_* {^p}p_1^! \mcm &\simeq j^! \omega_* {^p}p_1^! j_{0*} \pi^! \mcn \\
&\simeq j^! \omega_* (j_0 \times \id)_* {^p}\overline{p}_1^! \pi^! \mcn \\
&\simeq j^! \omega_* (j_0 \times \id)_* (\pi \times \id)^! {^p}\pi_1^! \mcn \\
&\simeq j^! \overline{\omega}_* (\pi \times \id)^! {^p}\pi_1^! \mcn \\
&\simeq \tilde{\omega}_* \tilde{j}^! (\pi \times \id)^! {^p}\pi_1^! \mcn\\
&\simeq \tilde{\omega}_* \pi_U^! j_U^! {^p}\pi_1^! \mcn \\
&\simeq (\id \times \xi)_* (f^{-1})_* \pi_U^! j_U^! {^p}\pi_1^! \mcn \\
&\simeq (\id \times \xi)_* f^! \pi_U^! j_U^! {^p}\pi_1^! \mcn \\
&\simeq (\id \times \xi)_*p_2^!g^!j_U^! {^p}\pi_1^! \mcn
\end{align*}
Set $\mcn':= g^!j_U^! \pi_1^! \mcn$. We have $(\id \times \xi)_*p_2^! \mcn' \simeq \tilde{p}^!_2\xi_* \mcn'$ where $\tilde{p}_2: \mbc^*_z \times V \ra V$ is the projection to the second factor. This  shows that $j^! \omega_* p_1^! \mcm \simeq \tilde{p}^!_2\xi_* \mcn'$ is constant in the $z$-direction. Hence the logarithm of the monodromy $N$ acts trivially on the (unipotent) nearby cycles $\psi_z \omega_* p_1^! \mcm$ and therefore also on the vanishing cycles $\phi_z \omega_* p_1^! \mcm$.

Set $L_i \phi_z \omega_* {^p}p_1^! \mcm := \phi_z W_i \mch^{0}
\omega_* {^p}p_1^! \mcm$. The weight filtration on $\phi_z \omega_*
      {^p}p_1^! \mcm = \phi_z \mch^0 \omega_* {^p}p_1^! \mcm$ is the
      relative monodromy weight filtration with respect to the
      filtration $L$ and the nilpotent endomorphism $N$. In the
      (current) case $N = 0$ we simply get $W_i  \phi_z \omega_* p_1^! \mcm = L_i \phi_z \omega_* p_1^! \mcm = \phi_z W_i \mch^0\omega_* p_1^! \mcm$
(cf. \cite[(2.2.7) \& Proposition 2.4]{SaitoMHM}).\\

We now want to prove by decreasing induction on $\ell$ that
\begin{itemize}
\item $W_{\ell-n} \phi_z \omega_* {^p}p_1^! \mcm = \phi_z \omega_* {^p}p_1^! W_\ell \mcm$ 
\item $\phi_z \omega_* {^p}p_1^! Gr^W_\ell \mcm$ is pure of weight $\ell-n$.
\end{itemize}
This is certainly true for $\ell \gg 0$ since in this case $W_{\ell} \mcm = \mcm$. Assume now that the two statements above are true for some $\ell$, we prove the two statements for $\ell-1$. For this consider the exact sequence
\begin{equation}\label{eq:omegaseq}
\phi_z \omega_* {^p}p_1^! W_{\ell-1} \mcm \lra \phi_z \omega_* {^p}p_1^! W_{\ell} \mcm \lra \phi_z \omega_* {^p}p_1^! Gr^W_{\ell} \mcm.
\end{equation}
Since $ W_{\ell-n}\phi_z \omega_* {^p}p_1^! \mcm = \phi_z \omega_*{^p}p_1^! W_{\ell}\mcm $ and since $\phi_z \omega_* {^p}p_1^! Gr^W_{\ell} \mcm$ is pure of weight $\ell-n$ we see that
\[
\phi_z \omega_* {^p}p_1^! W_{\ell-1} \mcm \supseteq W_{\ell-1-n} \phi_z \omega_* {^p}p_1^!  \mcm.
\]
To show the other inclusion, we consider the morphism
\begin{gather}\label{eqn-Iell-1}
\mch^0 \omega_! {^p}p_1^! W_{\ell-1} \mcm  \lra \mci_{\ell-1} \lra \mch^0 \omega_* {^p}p_1^! W_{\ell-1}\mcm,
\end{gather}
where $\mci_{\ell-1}$ is the image of the morphism $\mch^0 \omega_! {^p}p_1^! W_{\ell-1}\mcm  \ra \mch^0 \omega_* {^p}p_1^! W_{\ell-1}\mcm$. Notice that the map \eqref{eqn-Iell-1} becomes an isomorphism after applying $\phi_z$ (cf.\cite[equation 10.3.32]{KS}).\\

Since ${^p}p_1^!W_{\ell-1} \mcm = W_{\ell-1-n} {^p}p_1^! \mcm$ and the functor $\omega_!$ does not increase weight we have $\mch^0 \omega_! {^p}p_1^! W_{\ell-1}\mcm \subseteq W_{\ell-1-n} \mch^0 \omega_! {^p}p_1^! \mcm$. Because $\mci_{\ell-1}$ is a quotient of $\mch^0 \omega_! {^p} p_1^! W_{\ell-1}\mcm$ we also have $W_{\ell-1-n} \mci_{\ell-1} = \mci_{\ell-1}$. Since $\mci_{\ell-1}$ is a subobject of $\mch^0 \omega_* {^p}p_1^! W_{\ell-1}$ the morphism $N$ acts trivially and therefore $W_{\ell-1-n} \phi_z \mci_{\ell-1} = \phi_z \mci_{\ell-1}$. The isomorphism $\phi_z \mci_{\ell-1} \simeq \phi_z \omega_* {^p}p_1^! W_{\ell-1} \mcm$ shows 
\[
\phi_z \omega_* {^p}p_1^! W_{\ell-1} \mcm = W_{\ell-1-n}\phi_z \omega_* {^p}p_1^! W_{\ell-1} \mcm   \subseteq W_{\ell-1-n} \phi_z \omega_* {^p}p_1^!  \mcm
\]

We now want to show that $\phi_z \omega_* {^p}p_1^! Gr^W_{\ell-1} \mcm$ is pure of weight $\ell-1-n$. For this consider the morphisms
\[
\mch^0 \omega_! {^p}p_1^! Gr^W_{\ell-1} \mcm \lra \mcg_m \lra \mch^0 \omega_* {^p}p_1^! Gr^W_{\ell-1} \mcm
\]
where $\mcg_{\ell-1}$ is the  image of the morphism $\mch^0 \omega_! {^p}p_1^! Gr^W_{\ell-1} \mcm \ra \mch^0 \omega_* {^p}p_1^! Gr^W_{\ell-1} \mcm$. Notice again that the map above becomes an isomorphism after applying $\phi_z$. Since ${^p}p_1^!$ shifts weight by $-n$, and since $\omega_!$ does not increase weight and since $\omega_*$ does not decrease weight the module $\mcg_{\ell-1}$ is pure of weight $\ell-1-n$. Since $\phi_z \mcg_{\ell-1}$ is a subobject of $\phi_z \mch^0 \omega_* {^p}p_1^! Gr^W_{\ell-1} \mcm$ and $\phi_z \mch^0 \omega_* {^p}p_1^! Gr^W_{\ell-1} \mcm$ is a quotient of $\phi_z \mch^0 \omega_* {^p}p_1^! \mcm$ the morphism $N$ is trivial on $\phi_z \mcg_{\ell-1}$. Therefore $\phi_z \mcg_{\ell-1} \simeq \phi_z \mch^0 \omega_*{^p}p_1^! Gr^W_{\ell-1} \mcm$ is pure of weight $\ell-1-n$.

  This finishes the proof of the proposition.
\end{proof}

\begin{rem}
  The reader might note that \cite[Thm.~1.4]{ChenDirks} contains
  (amongst other things) related results on monodromic mixed Hodge
  modules and weight filtrations.
\end{rem}

If we endow the GKZ-system $\mcm^0_A$ with the mixed Hodge module
structure coming from the monodromic Fourier transformation we get the following result.

\begin{cor}\label{cor:monodromic}

  For homogeneous $A$ in the context of Corollary \ref{cor:hatW} and
  Definition \ref{defn:FStrafo}, let ${^H\!\!}\mcm^0_A$ be the
  GKZ-system endowed with the mixed Hodge module structure coming from
  the isomorphism
  \[
  \mcm^0_A \simeq \Dmod(\phi_z \omega_*{^p}p_1^!
  {}^H\!\!\hatmcm_A^0)
  \]
  with ${^H\!\!}\hatmcm^0_A$ as in Corollary \ref{cor-Mhat}.
  Then
\[
\Dmod (W_{k-n} {^H\!\!}\mcm^0_A) = \phi_z \omega_*{^p}p_1^!(W_k{^H\!\!}\hatmcm^0_A)=\FL(W_k\hatmcm^0_A).
\]
\end{cor}
\begin{proof}
It remains to shows that ${^H\!\!}\hatmcm^0_A$ can be written as $(j_0)_* \pi^! \mcn$ for some $\mcn \in \rmD^b \MHM(\mbp(\hatV))$. Consider the diagram
\[
\xymatrix{T \ar[dd]^{pr}  \ar[drr]^{\TVmap_0}\ar[rr]^\TVmap & & \hatV \\
& & \hatV \setminus \{0\}  \ar[u]^{j_0} \ar[d]^\pi \\
\overline{T} \ar[rr]^{\overline{\TVmap}} && \mbp(\hatV) }
\]
where $pr: T \ra \overline{T}$ is the projection to the last $d-1$ coordinates $\TVmap = j_0 \circ \TVmap_0$ is the canonical factorization and $\overline{\TVmap}$ is the projectivization of $\TVmap$. We have
\begin{align*}
{^H\!\!}\hatmcm^0_A \simeq \TVmap_* {^p}\mbq^H_T &\simeq \TVmap_* pr^*\, {^p}\mbq^H_{\overline{T}} \simeq \TVmap_* pr^!\, {^p}\mbq^H_{\overline{T}} \simeq \TVmap_* pr^! \, {^p}\mbq^H_{\overline{T}}(-1)[-1]\\
 &\simeq (j_0)_* (\TVmap_0)_*pr^! \, {^p}\mbq^H_{\overline{T}}(-1)[-1] \simeq (j_0)_* \pi^!  \underbrace{\overline{\TVmap}_*  {^p}\mbq^H_{\overline{T}}(-1)[-1] }_{=:\mcn}
\end{align*}
\end{proof}

\begin{rem}
In \cite{Reich1} a homogeneous GKZ-system was equipped with the structure of a mixed Hodge module by using the Radon transformation. At the moment we do not know if both mixed Hodge module structures coincide, but we believe that they coincide up to Tate twist.
\end{rem}

\section{Explicit weight filtration for $d=3$}\label{sec-d=3}

Throughout this section, $A$ is normal but not necessarily homogeneous. Via the Fourier transform
$\FL$ one can port the weight filtration on the mixed Hodge module
$\TVmap_*{}^p\QQ^H_T$ to the hypergeometric system $\calM_A^0$. While
the latter may not be a mixed Hodge module, one still obtains in any
case a filtration that has semisimple associated graded pieces and
which we still denote by $W_\bullet$. If $A$
is homogeneous, then $\calM_A^0$ is a mixed Hodge module and, by
Corollary \ref{cor:monodromic}, $\FL$ agrees with the functor
$\phi_z\omega_*{}^pp_1^!$ and relates the weight filtrations on
$\calM_A^0$ and $\TVmap_*{}^p\QQ^H_T$.  In this section we consider
specifically the cases when either $\NN A$ is simplicial, or when $d\le3$
and write out an explicit filtration in terms of generators that
agrees with $W_\bullet$.

Batyrev proved that in the homogeneous, normal case the weight
filtration on the restriction of $\mcm^0_A$ to the complement of the
principal $A$-discriminant is given by the face filtration on $S_A$ in
the sense that (in the localization) $W_{d+k}(\calM_A^0)$ is generated
by the $\del$-monomials whose degree sits in the relative interior of
a face of $\sigma$ whose codimension is at most $k$; see \cite[Thm.~8,
  p.28]{Sti}. It has been speculated that this be true even on
$\mcm^0_A$ itself. We show here that this is the case for simplicial
homogeneous $\sigma$ but can fail in the general homogeneous case
already in dimension three. We discuss completely in terms of
generators the filtration $\FL(W_\bullet \TVmap_*(\calO_T))$ if $d=3$
and $A$ is normal (but not necessarily homogeneous). Then $\sigma$ is
the cone over a $(d-1)$-dimensional polygon $P$ with $f_0$ vertices
and $P$ arises as intersection of $\sigma$ with a generic
hyperplane. It is not suggested or required that the columns of $A$
lie on $P$. It is sufficient to concentrate on the global sections
$M_A^0$.

\begin{ntn}\label{ntn-batyrev}
  On $M_A^0$, let $W'_\bullet$ be the filtration of Batyrev:
  \[
  W'_{d+k}(M_A^0)=\text{image of }D_A\cdot\left\{\del^\boldu\mid A\cdot \boldu=\bolda\in\Int(\NN\tau),
  \dim(\tau)\geq \dim(\sigma)-k\right\} \text{ in }M_A^0
  \]
for $d\le k\le 2d$. In particular, $W'_{< d}(M_A^0)=0$ and
$W_{>2d}(M_A^0)=M_A^0$. 
  
  For $d=3$ let $W''_\bullet$ be the filtration
  \[
  W''_k(M_A^0)=\left\{\begin{array}{rl}
  W'_k(M_A^0)&\text{ if } k\neq 2d-2,\\
  W_k'+\sum_r D_A\cdot e_r&\text{ if }k=2d-2,\end{array}\right.
  \]
  where $e_r$ is defined below in \eqref{eq-e_r}.
\end{ntn}

For ease of
notation , we do not repeat ``$M_A^0$'' each time we write a
filtration piece.
We will show that $W''=W$ if $d\le
3$, and that $W'=W''=W$ if $\sigma$ is simplicial.
For this,
consider the toric modules defined as follows.
\begin{ntn}
  If $\tau$ is a face of $\sigma$ write $\del_\tau^+$ for the $S_A$-ideal
generated by the $\del$-monomials whose degree is interior to $\tau$. 
  Let $S_A^{(k)}$ be the ideal of $S_A$
  spanned by the monomials that are interior to a face of codimension
  $k$ or less, $S_A^{(k)}=\sum_{\dim\tau\geq d-k}\del_\tau^+$. Then $S_A^{(0)}$ is the interior ideal, $S_A^{(d-1)}$ is
  the maximal ideal $S_A\del_A$, and $S_A^{(d)}$ is $S_A$ itself.
\end{ntn}
We begin with showing that for normal $S_A$ the $D_A$-module generated by the interior
ideal $S_A^{(0)}$ inside $M_A^0$ is simple and for homogeneous $A$ agrees with
$W_d$ so that $W_k=W'_k=W''_k$ for $k\le d$.

\begin{lem}\label{lem:int}
  Suppose $A$ is pointed and saturated, but not necessarily 
  homogeneous. Let $\boldu,\boldv\in\NN^n$ be
  such that $\boldb:=A\cdot\boldv$ is in the interior $\Int(\NN A)$ of
  the semigroup (\emph{i.e.}, not on a proper face). Set
  $\bolda=A\cdot\boldu$. Then the contiguity
  map $c_{-\bolda-\boldb,-\boldb}\colon M_A^{-\bolda-\boldb}\stackrel{\cdot\del^\boldu}{\to} M_A^{-\boldb}$
  is an  isomorphism.

  In particular, the ideal in $M_A^0$ generated by $\del^\boldb$ (
  the image of the contiguity morphism $c_{-\boldb,0}\colon
  M_A^{-\boldb}\to M_A^0$)  is the same for all $\boldb=A\cdot \boldv$
  in the interior of $A$. 
\end{lem}
\begin{proof}
  Consider the toric sequence $0\to
  S_A(\bolda)\stackrel{\cdot\del^\boldu}{\to} S_A\to
  Q:=S_A/S_A\cdot\del^\boldu\to 0$, and the Euler--Koszul functor
  attached to $-\boldb$.  By \cite[Prop.~5.3]{MMW05}, the induced
  contiguity morphism $c_{-\boldb-\bolda,-\boldb}\colon
  M_A^{-\boldb-\bolda}\to M_A^{-\boldb}$ is an isomorphism if and only
  if $-(-\boldb)$ is not quasi-degree of $Q$.  The quasi-degrees of
  $Q=S_A/\del^\boldu S_A$ are contained in a union of hyperplanes that
  meet $-\NN A$ and are parallel to a face of the cone $\sigma$. In
  particular, these quasi-degrees are disjoint to the interior points
  $\Int(\NN A)\ni\boldb$ of $\NN A$. It follows that
  $c_{-\boldb-\bolda,-\boldb}$ is an isomorphism for all $\boldb\in
  \Int(\NN_A)$.

  Now consider the composition
  \[
  c_{-\boldb,0}\circ c_{-\bolda-\boldb,-\boldb}\colon M_A^{-\bolda-\boldb}\to M_A^{-\boldb}\to M_A^0 
  \]
  with $\bolda,\boldb\in\Int(\NN A)$. The first map is an isomorphism, and so the image of the composition is just the image of $c_{-\boldb,0}$. For any two elements $\boldb,\boldb'\in\Int(\NN A)$, factoring $M_A^{-\boldb-\boldb'}\!\to M_A^0$ through $M_A^{-\boldb}$ or $M_A^{-\boldb'}$ shows that the images of $c_{-\boldb,0}$ and $c_{-\boldb',0}$ agree with the image of $c_{-\boldb-\boldb',0}$. In particular, they are equal. Since $\del^\boldu$ is in the image of $c_{-A\cdot\boldu,0}$,
    the image of $c_{-\boldb,0}$ contains all of $\Int(\NN A)$ whenever
    $\boldb\in\Int(\NN A)$. 
%
%
\end{proof}
It follows that for normal $S_A$ the submodule of $M_A^0$ generated by
any interior monomial of $S_A$ agrees with that submodule generated by
$S_A^{(0)}$. If $A$ is homogeneous, so that $\FL$ carries the mixed
Hodge module structure from $\TVmap_*{}^p\QQ^H_T$ to $\calM_A^0$,
the level $d$ part of $W_\bullet$ has the property that
$\TVmap_*{}^p\QQ^H_T/\FL^{-1}(W_d\calM_A^0)$ is supported on the
boundary tori. Thus, any section of this sheaf is killed by some power
of $x_1\cdots x_n$, so that each element of $M_A^0/W_d$ is killed by
some power of $\del_1\cdots \del_n$. That means that $W_d$
contains (the coset of) an interior monomial of $S_A$, and hence $W_d$
contains the submodule generated by $S_A^{(0)}$. Since $W_d$ is
simple, it cannot strictly contain it, so must be equal to
it.

As an aside, note that the Euler--Koszul homology module
$H_0^A(S_A^{(0)};0)$ associated to the interior ideal is the
underlying $D_A$-module to
$\FL(\TVmap_\dag{}^p\QQ^H_T)=\FL\DD(\TVmap_*{}^p\QQ^H_T)$. Indeed, it
follows from \cite{Walther1} that the dual of $M_A^0$ is
$M_A^{-\gamma}$ for some interior point of $\NN A$. Since $S^{(0)}_A$
is the direct limit of all principal ideals generated by interior
monomials, $H_0^A(S_A^{(0)};0)$ is the direct limit of all
$H_0^A(S_A\cdot\del^\boldu;0)$ with $\del^\boldu$ interior to $\NN
A$. It follows from Lemma \ref{lem:int} that the structure morphisms in the
limit are all isomorphisms. Thus, we can identify the morphisms
$H_0^A(\Int(\NN A);0)\to M_A^0$ and $\FL \Dmod\left(\TVmap_!{}^p\QQ^H_T\to
\TVmap_*{}^p\QQ^H_T\right)$, and the corresponding statement holds for any
face $\tau$ with lattice  $\tau_\ZZ$.

It is clear that $S_A^{(k)}\subseteq S_A^{(k+1)}$ and that the
quotient $S_A^{(k)}/S_A^{(k-1)}$ is the direct sum of the interior
ideals of the face rings $S_\tau$ for which $\dim(\tau)=d-k$. It follows
that $\gr^{W'}_k(M_A^0)$ surjects onto each $H_0^A(\Int(S_\tau);0)$,
the Euler--Koszul module defined over $D_A$ by the toric module formed
by the graded maximal submodule of the toric module
$S_\tau=S_A/\{\del_j\mid j\notin\tau\}$, see \cite{MMW05} for details. It therefore also surjects
onto the image of $H_0^A(\Int(S_\tau);0)$ in $H_0^A(S_\tau;0)$, the
underlying $D_A$-module corresponding to $\IC_{X_\tau}$ under the
monodromic Fourier transform.

If $\NN A$ is simplicial, Theorem \ref{thm:weightonhQ} implies that
$\gr^W_{d+k}(\TVmap_*{}^p\QQ^H_T)$ is the sum of intersection
complexes $\IC_{X_\tau}$ with $\dim(\tau)+k=d$, and each appears with
multiplicity one. Thus, $\gr^{W'}_k(M_A^0)$ surjects onto
$\gr^W_k(M_A^0)$ for all $k$ when $\NN A$ is simplicial. (We are not
asserting that this surjection is induced from a filtered morphism,
only that there is one; after all, we don't know $W$ at this
point). But $M_A^0$ is holonomic and by the Jordan--H\"older property
this implies that $W'=W$ when $\NN A$ is simplicial. This recovers for
$\beta\in\NN A$ a result of \cite{Fang}.

\medskip

Now suppose $d=3$ but don't assume simpliciality. (If $d\le
  2$, $\NN A$ is always simplicial).  By Theorem \ref{thm:weightonhQ}
any composition chain for $M_A^0$ will (up to Fourier transform)
have as composition factors exactly one copy of the intersection complex to $\tau$ for
$\dim(\tau)>0$, and $1+f_0-d$ copies of $\IC_0$. This means that an
epimorphism $\gr^{W''}(M_A^0)\to\gr^W(M_A^0)$ alone will not be enough to show $W''=W$
since the copies of $\IC_0$ need to be shown to live in the
right levels.

In any event, $W_6=M_A^0$ and $W_3$ is generated by the
interior ideal $S_A^{(0)}$. Equivariance and the fact that
$\gr^W_6$ must equal $\CC[x_A]$ shows that $W_5$ is
generated by the maximal ideal $S_A^{(2)}$. It remains to find
generators for $W''_4$, such that  there are surjections
$\gr^{W''}_k(M_A^0)\to \gr^W_k(M_A^0)$ for $k=4,5$ such that at least
one is an isomorphism. 

For arbitrary saturated $\NN A$ with $d=3$, define on
  $M_A^0=D_A/(I_A,E)$ a filtration as follows:
  \begin{itemize}
  \item $W''_i=0$ for $i<3$;
  \item $W''_3$ is the left ideal generated by $\del^+_\sigma$;
  \item $W''_4$ is the left ideal generated by $W''_3$ and all $\del^+_{\tau_2}$
    where $\dim(\tau_2)=2$, plus the left ideal generated by all $e_r$
    defined below, where $e$ runs through the $f_0$ vertices of $P$;
  \item $W''_5$ is the left ideal generated by $W''_4$ and all
    $\del^+_{\tau_1}$ where $\dim(\tau_1)=1$;
  \item $W''_6$ is the left ideal generated by $1\in D$.
  \end{itemize}

  We now describe the operators $e_r$. Choose \emph{distinguished}
  nonzero columns $\{\boldb_r\}_1^{f_0}$ of $A$ that correspond to the
  primitive lattice points on the rays through the vertices of the
  polygon $P$ (which are in $A$ since $\NN A$ is saturated). 

  For each distinguished $\boldb_r$ define a function $F_r$ on
  $A=\{\bolda_1,\ldots,\bolda_n\}$ as follows:
  \begin{eqnarray}\label{eqn-F}
    F_r(\bolda_j)&=&\left\{\begin{array}{cl}
    1&\text{if } \bolda_j=\boldb_r;\\
    c_{r,j}&\text{if } \bolda_j=c_{r,j}\cdot\boldb_r+
    c_{r',j}\cdot\boldb_{r'}\text{ is on the $\sigma$-face
      spanned by $\boldb_r$ and
      $\boldb_{r'}$;}\\
    0&\text{else.}\end{array}\right.
  \end{eqnarray}
  We call \emph{invisible from $\boldb_r$} any $\bolda\in\ZZ A$ for
  which the ray from $\boldb_r$ to $\bolda$ passes through the
  interior of $\sigma$. Then $F_r$ vanishes on all $\bolda_j$
  invisible from $\boldb_r$, and $F_r$ is piece-wise linear
  on the $2$-faces of $\sigma$ (which are in bijection with edges of
  $P$). Set
    \begin{gather}\label{eq-e_r}
  e_r=\sum_jF_r(\bolda_j)x_j\del_j.
  \end{gather}

We now show that our filtration $W''$ is indeed the Fourier--Laplace
transform of the weight
filtration on $\TVmap_*{}^p\QQ^H_T$.
Note first that $W''_5$ indeed contains $W''_4$
  (specifically, the $e_r$).
  We prove now, that each $e_r$ is annihilated by $\frakm$ in
  $M_A^0/W_3$, and hence they are candidates for the intersection complexes
  in $W_4/W_3$ with support in $0$.
  
  Let $\boldb_{r_1}$ and $\boldb_{r_2}$ be the two distinguished
  columns that lie on a facet with $\boldb_r$. Then there is a unique
  linear function $E_r$ on $\RR^3$ whose values agree with those of
  $F_r$ on $\boldb_r, \boldb_{r_1}$ and $\boldb_{r_2}$. We denote the
  corresponding Euler operator also by $E_r$. The linearity of $F_r$
  along facets implies that $F_r$ and $E_r$ agree on all $\bolda_j$
  that have $F_r(\bolda_j)$ nonzero (which are the $\bolda_j$ not
  invisible from $\boldb_r$). Thus, in
  $M_A^0/W''_3=D/(I_A,E,\del_\sigma^+)$, the expression $e_r$ is
  equivalent to a linear combination $e_{r,0}=e_r-E_r$ of
  $\{x_j\del_j\}_j$ for which each $\bolda_j$ with nonzero coefficient
  is invisible from $\boldb_r$. Now, in $D/(I_A,E,\del_\sigma^+)$, $\del_je_r$
  is zero for $\bolda_j$ invisible from $\boldb_r$, and
  $\del_je_{r,0}=0$ also for $\bolda_j$ any integer multiple of
  $\boldb_r$ and for $\bolda_j$ interior to the facets touching
  $\boldb_r$.  If $\bolda_j=\boldb_{r_1}$, consider the Euler operator
  $E_{r_1}$ that agrees with $e_r$ on $\boldb_r$ and on
  $\boldb_{r_1}$, and which takes value zero on the $2$-face of
  $\sigma$ containing $\boldb_{r_1}$ but not $\boldb_r$. Then
  $(e_r-E_{r_1})$ has all terms invisible from $\bolda_j$ and so
  $\del_j(e_r-E_{r_1})$ is zero in $W''_6/W''_3$. A similar argument
  works for $\bolda_j=\boldb_{r_2}$. Hence every $\del_j$ annihilates
  the class of $e_r$ in $W''_6/W''_3$ and so $e_r$ spans a module in
  $W''_6/W''_3$ that is either zero or $D/D\frakm$. Note that there
  are $d=\dim(\sigma)=3$ linear dependencies between the cosets of the
  $e_r$ in $M_A^\beta$, so the $\{e_r\}_r$ are spanning a module
  isomorphic to a submodule of $\oplus_1^{f_0-d}D/D\frakm$.

  Next, let $\bolda$ be in the relative interior of a facet $\tau$ of
  $\sigma$. Since $W_3$ contains every interior monomial of $\sigma$,
  the coset of $\del^\bolda$ in $M_A^0/W''_3$ is $\del_j$-torsion for
  all $j\not\in\tau$. Let $\TVmap_\tau\colon T_\tau\to\CC^\tau$ be the
  toric map, induced by the restriction of $A$ to $\tau$, from the
  $\tau$-torus to the subspace $\CC^\tau$ of $\CC^A$ parameterized by
  the columns of $A\cap\tau$. The submodule generated by $\del^\bolda$
  inside $M_A^0/W''_3$ is isomorphic to a quotient of
  the simple module $\CC[x_{\tau^c}]\otimes_\CC
  \FL\image((\TVmap_\tau)_\dag\to(\TVmap_\tau)_+)$, where $x_{\tau^c}$ are
  the $x_j$ with $j\not\in\tau$. 

  Now
  consider an interior monomial $\del^\bolda$ of a ray $\tau_1$ of
  $\sigma$. Then in $M_A^0/W''_4$, $\del^\bolda$ is killed by all
  $\del_j$ with $j\not\in \tau_1$.
  Modulo the $\del_j$ not sitting on any ray, $e_r$ becomes exactly the
  Euler operator for $M_{\tau_1}^0$ if $\boldb_r$ sits on $\tau_1$, and
  hence (after the Fourier transform) the module generated by $\del^\bolda$ in $M_A^0/W''_4$ is
  exactly the intersection complex associated to $\tau_1$ (pushed to
  $V$). Hence 
  $W''_5/W''_4\simeq\bigoplus_{\tau_1}\IC_{\tau_1}$, and so 
  %
  \begin{itemize}
  \item $W_k=W''_k$ if $k\le 3$ and if $k\geq 5$;
  \item $W''_5/W''_4\simeq W_5/W_4$;
  \item hence $W''_4/W''_3\simeq W_4/W_3$ by Jordan--H\"older.
  \end{itemize}
  Since the faces whose intersection complexes appear as summands in
  $W_5/W_4$ have dimension one, and those in
  $W_4/W_3$ have dimension $0$ or $2$, $W''_4$ must equal $W_4$. 

  \begin{ex}
    Let $A=\begin{pmatrix}1&1&1&1\\0&1&0&1\\0&0&1&1\end{pmatrix}$, one
    of the possible matrices whose GKZ-system (with the right $\beta$)
    contains Gau\ss' hypergeometric $\!{}_2F_1$ as solution. We have
    $n=4$ and $d=3$, and $P$ is a square in which $1$ and $4$ are
    opposite vertices. The Euler space $E$ is spanned by
    $x_1\del_1+x_3\del_3$, $x_2\del_2+x_4\del_4$ and
    $x_3\del_3+x_4\del_4$. The four elements $e_r$ are simply
    $\{x_j\del_j\}_1^4$. The toric ideal is generated by
    $\del_1\del_4-\del_2\del_3$. The interior ideal of $S_A$ is
    generated by $\del_2\del_3$. The weight filtration on $M_A^0$ is
    given by $W_2=\overline{0}$,
    $W_3=\overline{\{E,\del_1\del_4,\del_2\del_3\}}$,
    $W_4=W_3+\overline{\{\del_1\del_2,\del_2\del_4,\del_4\del_3,\del_3\del_1\}}+
      \overline{\{e_1,e_2,e_3,e_4\}}$,
    $W_5=W_4+\overline{\{\del_1,\del_2,\del_3,\del_4\}}$,
    $W_6=M_A^0$. Here, the bar indicates taking cosets on $M_A^0$. Note that the three Euler dependencies in $H_A(0)$
    imply that the four operators $e_r$ generate only one copy of
    $\IC_0$ inside $W_4/W_3$.
  \end{ex}
  
  \section{Concluding remarks and open problems}

  \begin{asparaenum}
    \item We assume throughout that $S_A$ is normal, which covers the
  most significant geometric situations. One obvious challenge is to
  remove this hypothesis and generalize our results. This would be
  likely difficult since then arithmetic issues will enter the fray.

  \item In another direction it would be interesting to see what can
    be done (as mixed Hodge module or otherwise) when $\beta\not
    =0$. In an article of Fang, composition chains for hypergeometric
    systems are considered that are based on the filtration-by-faces
    on the semigroup ring, see \cite{Fang} and refer to \cite{AS} for
    motivating discussion. This filtration (see Notation
    \ref{ntn-batyrev}) was first considered by Batyrev in \cite{Bat4},
    but see also \cite{Sti}. (We note in passing that the
      filtration-by-faces is not a natural filtration: typically, if
      GKZ-systems $M_A^\beta \simeq M_A^\gamma$ are isomorphic under a
      contiguity morphism, the two face filtrations do not
      correspond).  The hypotheses are somewhat technical, but in the
    simplicial normal case \cite{Fang} shows essentially that for
    $\beta=0$ the filtration-by-faces gives semisimple composition
    factors. Comparing with the weight filtration, this corresponds to
    all non-diagonal terms $\mu^\sigma_\tau(e)$ with
    $\dim(\tau)+e\not=d$ being zero in Theorem \ref{thm:weightonhQ},
    the case of trivial combinatorics in the polytope to $\sigma$.

  \item By adding all nonzero $\mu^\sigma_\tau(e)$ one obtains the
    holonomic length of $M_A^0$. Is there a compact formula? In
    particular, does it give a better estimate than the general
    exponential bounds in \cite{SST}? When $P$ is simplicial,
    $\ell(M_A^0)=2^d$, while for $d=3,4,5$ these lengths are for
    general $P$ as follows, where in generalization of the face
    numbers $f_i$ of $P$ we denote $f_{i,j}$ the number of all pairs
    ($i$-face, $j$-face) that are contained in one another. For
    relations between the various $f_{i,j}$ for $4$-polytopes, see
    \cite{Bayer-4polytopes}. 
    \[
    \begin{array}{||c|ll||}\hline
      d&&\ell(M_A^0)\\\hline\hline
      3&&1+f_0+f_1+f_2+(f_0-3)=3f_0-1\\\hline
      4&&1+f_0+f_1+f_2+f_3+(f_0-4)+(f_{1,0}-3f_0)\\
      &=&-2f_0+4f_1\\\hline
      5&&1+f_0+f_1+f_2+f_3+f_4+(f_0-5)+(f_{1,0}-4f_0)+(f_{2,1}-3f_1)
      +(f_{2,0}-3f_2+f_1-4f_0+10)\\
      &=&7-5f_0-f_2+2f_{2,0}\\\hline\hline\end{array}
      \]
      Of course, all these numbers are non-negative. Is there an
      obviously non-negative representation that is more intelligible
      than $\sum\mu^\sigma_\tau(e)$?
    
\item Given $A$ and a face $\tau$, what is the holonomic rank of the
  Fourier transform of the intersection complex on the orbit to
  $\tau$? Such formul\ae\ would be very interesting even for normal
  simplicial $A$ since it interweaves volume-based expressions for
  rank with combinatorial expressions in the way Pick's theorem talks
  about polygons. For example, when $d=2$ and $A$ is normal, one can
  derive from our results that the rank of $\FL(\mcm^{\IC}(X_\tau))$
  always differs from the volume of $A$ by one. Induction on $d$ gives
  recursions, but an explicit formula is unknown.

\item In Section \ref{sec-d=3} we explained how to write down
  explicitly the weight filtration for $d=3$. For $d=4$, similar ideas
  can be used to write out explicit generators. But starting with
  $d=5$ this seems a very hard problem. Part of the issue is that
  writing down such filtration would produce a non-canceling
  expression for the higher intersection cohomology dimension of
  polytopes of dimension 4 or greater, which we do not think are known.
\end{asparaenum}

\section*{List of symbols}
\begin{itemize}
\item $\CC^n=V=\Spec\CC[x_1,\ldots,x_n]$, the domain of the GKZ system $M_A^\beta$,
\item $\CC^n=\hatV=\Spec[y_1,\ldots,y_n]$, the target of $\TVmap$, 
\item $T$ the $d$-torus,
\item $\TVmap\colon T\to \hatV$ the monomial map induced by $A$,
\item $X$ the closure of $T$ in $\hatV$,
\item $\TXmap\colon T\to X$ the restriction of $\TVmap$,
\item $\deg(x)=\bolda=-\deg(\del)$ the $A$-degree function on
  $S_A=\CC[\NN A]$,
\item $\phi_z$ the vanishing cycle along the function  $z$,
\item $\psi_z$ the corresponding  nearby cycle,
\item $\XVmap\colon X\to V$ the closed embedding,
\item $i_\tau\colon \frakx_\tau\into X_\tau$ the embedding of the
  $T$-fixed point,
\item $\rho_\tau\colon \frakx_{\sigma/\tau}\times T_\tau
  \into X_\sigma$ and $j_\tau\colon X_{\sigma/\tau}\times T_\tau\into X_\sigma$
from $\NN A\to  (\sigma/\tau)_\NN\oplus \tau_\ZZ$,
\item the relative version $j_\tau^\gamma\colon X_{\gamma/\tau}\times
  T_\tau\to X_\gamma$ to $j_\tau$,
\item $i_{\tau,\gamma}\colon T_\tau\to X_\tau\to X_\gamma$ from
  $\gamma_\NN\onto \tau_\NN\to\tau_\ZZ$,
\item $\kappa_\boldv\colon \mbg_m = \Spec \mbc[z^\pm] \ra T = \Spec
  \mbc[\mbz A]$ the monomial action induced by $\boldv$,
\item $u_\tau\colon X_\tau\smallsetminus \frakx_\tau\into X_\tau$,
 \item $\XVmap_\tau\colon \frakx_\tau\to V$.
\end{itemize}

\bibliographystyle{amsalpha}
\bibliography{weight}

\def\cprime{$'$}
\providecommand{\bysame}{\leavevmode\hbox to3em{\hrulefill}\thinspace}
\providecommand{\MR}{\relax\ifhmode\unskip\space\fi MR }
\providecommand{\MRhref}[2]{%
  \href{http://www.ams.org/mathscinet-getitem?mr=#1}{#2}
}
\providecommand{\href}[2]{#2}
\begin{thebibliography}{dCMM18}

\bibitem[AS]{AS}
Alan Adolphson and Steve Sperber, \emph{{Composition series for $M_A(\beta)$}},
  Conference talk, accessed under
  \url{http://www.math.tamu.edu/~laura/pages/organizing/AMS-SLC-2011/adolphson.pdf}
  on September 21, 2018.

\bibitem[Bat93]{Bat4}
Victor~V. Batyrev, \emph{Variations of the mixed {H}odge structure of affine
  hypersurfaces in algebraic tori}, Duke Math. J. \textbf{69} (1993), no.~2,
  349--409.

\bibitem[Bay87]{Bayer-4polytopes}
Margaret Bayer, \emph{The extended {$f$}-vectors of {$4$}-polytopes}, J.
  Combin. Theory Ser. A \textbf{44} (1987), no.~1, 141--151. \MR{871395}

\bibitem[BBD82]{BBD}
A.~A. Be{\u\i}linson, J.~Bernstein, and P.~Deligne, \emph{Faisceaux pervers},
  Analysis and topology on singular spaces, {I} ({L}uminy, 1981), Ast\'erisque,
  vol. 100, Soc. Math. France, Paris, 1982, pp.~5--171. \MR{751966}

\bibitem[BK91]{BayerKlapper-cdIndex}
Margaret~M. Bayer and Andrew Klapper, \emph{A new index for polytopes},
  Discrete Comput. Geom. \textbf{6} (1991), no.~1, 33--47. \MR{1073071}

\bibitem[BM99]{BraMac}
Tom Braden and Robert MacPherson, \emph{Intersection homology of toric
  varieties and a conjecture of {K}alai}, Comment. Math. Helv. \textbf{74}
  (1999), no.~3, 442--455. \MR{1710686}

\bibitem[Bry86]{Brylinski}
Jean-Luc Brylinski, \emph{Transformations canoniques, dualit\'e projective,
  th\'eorie de {L}efschetz, transformations de {F}ourier et sommes
  trigonom\'etriques}, Ast\'erisque (1986), no.~140-141, 3--134, 251,
  G{\'e}om{\'e}trie et analyse microlocales.

\bibitem[CD21]{ChenDirks}
Qianyu Chen and Bradley Dirks, \emph{{On $V$-filtration, Hodge filtration and
  Fourier transform}}, Preprint arXiv:2111.04622.

\bibitem[CDK21]{CautisDoddKamnitzer}
Sabin Cautis, Christopher Dodd, and Joel Kamnitzer, \emph{Associated graded of
  {H}odge modules and categorical {$\frak{sl}_2$} actions}, Selecta Math.
  (N.S.) \textbf{27} (2021), no.~2, Paper No. 22, 55. \MR{4244328}

\bibitem[CLS11]{CLS}
David~A. Cox, John~B. Little, and Henry~K. Schenck, \emph{Toric varieties},
  Graduate Studies in Mathematics, vol. 124, American Mathematical Society,
  Providence, RI, 2011. \MR{2810322}

\bibitem[dCM09]{CataldoMigliorini-Bulletin}
Mark Andrea~A. de~Cataldo and Luca Migliorini, \emph{The decomposition theorem,
  perverse sheaves and the topology of algebraic maps}, Bull. Amer. Math. Soc.
  (N.S.) \textbf{46} (2009), no.~4, 535--633. \MR{2525735}

\bibitem[dCMM18]{deCataldoMiglioriniMustata}
Mark~Andrea de~Cataldo, Luca Migliorini, and Mircea Musta\c{t}\u{a},
  \emph{Combinatorics and topology of proper toric maps}, J. Reine Angew. Math.
  \textbf{744} (2018), 133--163. \MR{3871442}

\bibitem[Dim92]{DimcaHypersurface}
Alexandru Dimca, \emph{Singularities and topology of hypersurfaces},
  Universitext, Springer-Verlag, New York, 1992.

\bibitem[Dim04]{Di}
\bysame, \emph{Sheaves in topology}, Universitext, Springer-Verlag, Berlin,
  2004.

\bibitem[DL91]{DF1}
J.~Denef and F.~Loeser, \emph{Weights of exponential sums, intersection
  cohomology, and {N}ewton polyhedra}, Invent. Math. \textbf{106} (1991),
  no.~2, 275--294.

\bibitem[Fan20]{Fang}
Jiangxue Fang, \emph{Composition series for {GKZ}-systems}, Trans. Amer. Math.
  Soc. \textbf{373} (2020), no.~5, 3445--3481. \MR{4082244}

\bibitem[Fie91]{Fies}
Karl-Heinz Fieseler, \emph{Rational intersection cohomology of projective toric
  varieties}, J. Reine Angew. Math. \textbf{413} (1991), 88--98. \MR{1089798}

\bibitem[Ful93]{Fulton}
William Fulton, \emph{Introduction to toric varieties}, Annals of Mathematics
  Studies, vol. 131, Princeton University Press, Princeton, NJ, 1993, The
  William H. Roever Lectures in Geometry.

\bibitem[Gin86]{Ginsburg-Inv86}
V.~Ginsburg, \emph{Characteristic varieties and vanishing cycles}, Invent.
  Math. \textbf{84} (1986), no.~2, 327--402. \MR{833194}

\bibitem[Giv96]{Givental1}
Alexander~B. Givental, \emph{Equivariant {G}romov-{W}itten invariants},
  Internat. Math. Res. Notices (1996), no.~13, 613--663. \MR{1408320}

\bibitem[Giv98]{Givental2}
Alexander Givental, \emph{A mirror theorem for toric complete intersections},
  Topological field theory, primitive forms and related topics ({K}yoto, 1996),
  Progr. Math., vol. 160, Birkh\"auser Boston, Boston, MA, 1998, pp.~141--175.
  \MR{1653024}

\bibitem[GKZ90]{GKZ-Euler}
I.~M. Gel'fand, M.~M. Kapranov, and A.~V. Zelevinsky, \emph{Generalized {E}uler
  integrals and {$A$}-hypergeometric functions}, Adv. Math. \textbf{84} (1990),
  no.~2, 255--271. \MR{1080980}

\bibitem[Hoc72]{Hoch}
Melvin Hochster, \emph{Rings of invariants of tori, {C}ohen-{M}acaulay rings
  generated by monomials, and polytopes}, Ann. of Math. (2) \textbf{96} (1972),
  318--337.

\bibitem[HTT08]{Hotta}
Ryoshi Hotta, Kiyoshi Takeuchi, and Toshiyuki Tanisaki,
  \emph{{$\mathcal{D}$}-modules, perverse sheaves, and representation theory},
  Progress in Mathematics, vol. 236, Birkh\"auser Boston Inc., Boston, MA,
  2008, Translated from the 1995 Japanese edition by Takeuchi.

\bibitem[Iri09]{Iritani}
Hiroshi Iritani, \emph{An integral structure in quantum cohomology and mirror
  symmetry for toric orbifolds}, Adv. Math. \textbf{222} (2009), no.~3,
  1016--1079. \MR{2553377}

\bibitem[IX16]{IritaniXiao}
Hiroshi Iritani and Jifu Xiao, \emph{Extremal transition and quantum
  cohomology: examples of toric degeneration}, Kyoto J. Math. \textbf{56}
  (2016), no.~4, 873--905. \MR{3568645}

\bibitem[Kir86]{Ki}
Frances Kirwan, \emph{Rational intersection cohomology of quotient varieties},
  Invent. Math. \textbf{86} (1986), no.~3, 471--505. \MR{860678}

\bibitem[KS94]{KS}
Masaki Kashiwara and Pierre Schapira, \emph{Sheaves on manifolds}, Grundlehren
  der Mathematischen Wissenschaften [Fundamental Principles of Mathematical
  Sciences], vol. 292, Springer-Verlag, Berlin, 1994, With a chapter in French
  by Christian Houzel, Corrected reprint of the 1990 original.

\bibitem[KS97]{KS97}
\bysame, \emph{Integral transforms with exponential kernels and {L}aplace
  transform}, J. Amer. Math. Soc. \textbf{10} (1997), no.~4, 939--972.
  \MR{1447834}

\bibitem[LW19]{LW-equiDcat}
Andr\'{a}s~C. L\H{o}rincz and Uli Walther, \emph{On categories of equivariant
  {$\mathcal{D}$}-modules}, Adv. Math. \textbf{351} (2019), 429--478.
  \MR{3952575}

\bibitem[MMW05]{MMW05}
Laura~Felicia Matusevich, Ezra Miller, and Uli Walther, \emph{Homological
  methods for hypergeometric families}, J. Amer. Math. Soc. \textbf{18} (2005),
  no.~4, 919--941 (electronic).

\bibitem[Rei09]{Reich1}
Thomas Reichelt, \emph{A construction of {F}robenius manifolds with logarithmic
  poles and applications}, Comm. Math. Phys. \textbf{287} (2009), no.~3,
  1145--1187.

\bibitem[{Rei}14]{Reich2}
Thomas {Reichelt}, \emph{{Laurent Polynomials, GKZ-hypergeometric Systems and
  Mixed Hodge Modules}}, Compositio Mathematica \textbf{(150)} (2014),
  911--941.

\bibitem[RS17]{ReiSe2}
Thomas Reichelt and Christian Sevenheck, \emph{Non-affine {L}andau-{G}inzburg
  models and intersection cohomology}, Ann. Sci. \'{E}c. Norm. Sup\'{e}r. (4)
  \textbf{50} (2017), no.~3, 665--753. \MR{3665553}

\bibitem[RS20]{ReiSe3}
\bysame, \emph{Hypergeometric {H}odge modules}, Algebr. Geom. \textbf{7}
  (2020), no.~3, 263--345. \MR{4087862}

\bibitem[RSSW21]{RSSW}
Thomas Reichelt, Mathias Schulze, Christian Sevenheck, and Uli Walther,
  \emph{Algebraic aspects of hypergeometric differential equations}, Beitr.
  Algebra Geom. \textbf{62} (2021), no.~1, 137--203. \MR{4249859}

\bibitem[RSW18]{RSW18}
Thomas Reichelt, Christian Sevenheck, and Uli Walther, \emph{On the
  b-{F}unctions of {H}ypergeometric {S}ystems}, Int. Math. Res. Not. IMRN
  (2018), no.~21, 6535--6555. \MR{3873536}

\bibitem[Sai90]{SaitoMHM}
Morihiko Saito, \emph{Mixed {H}odge modules}, Publ. Res. Inst. Math. Sci.
  \textbf{26} (1990), no.~2, 221--333.

\bibitem[Sai94]{SaitoOnMHM}
\bysame, \emph{On the theory of mixed {H}odge modules}, Selected papers on
  number theory, algebraic geometry, and differential geometry, Amer. Math.
  Soc. Transl. Ser. 2, vol. 160, Amer. Math. Soc., Providence, RI, 1994,
  pp.~47--61.

\bibitem[SS]{Sabbah-webpage}
Claude Sabbah and Christian Schnell, \emph{The {MHM}-project}.

\bibitem[SST00]{SST}
Mutsumi Saito, Bernd Sturmfels, and Nobuki Takayama, \emph{Gr\"obner
  deformations of hypergeometric differential equations}, Algorithms and
  Computation in Mathematics, vol.~6, Springer-Verlag, Berlin, 2000.
  \MR{1734566 (2001i:13036)}

\bibitem[Sta87]{Stanley}
Richard Stanley, \emph{Generalized {$H$}-vectors, intersection cohomology of
  toric varieties, and related results}, Commutative algebra and combinatorics
  ({K}yoto, 1985), Adv. Stud. Pure Math., vol.~11, North-Holland, Amsterdam,
  1987, pp.~187--213. \MR{951205}

\bibitem[Sta92]{Stanley2}
Richard~P. Stanley, \emph{Subdivisions and local {$h$}-vectors}, J. Amer. Math.
  Soc. \textbf{5} (1992), no.~4, 805--851. \MR{1157293}

\bibitem[Ste19a]{Avi}
Avi Steiner, \emph{{$A$}-hypergeometric modules and {G}auss-{M}anin systems},
  J. Algebra \textbf{524} (2019), 124--159. \MR{3904304}

\bibitem[Ste19b]{Avi2}
\bysame, \emph{Dualizing, projecting, and restricting {GKZ} systems}, J. Pure
  Appl. Algebra \textbf{223} (2019), no.~12, 5215--5231. \MR{3975063}

\bibitem[Sti98]{Sti}
Jan Stienstra, \emph{Resonant hypergeometric systems and mirror symmetry},
  Integrable systems and algebraic geometry ({K}obe/{K}yoto, 1997), World Sci.
  Publ., River Edge, NJ, 1998, pp.~412--452.

\bibitem[SW09]{SchulzeWalther-ekdi}
Mathias Schulze and Uli Walther, \emph{Hypergeometric $\mathcal{D}$-modules and
  twisted {G}au\ss-{M}anin systems}, J. Algebra \textbf{322} (2009), no.~9,
  3392--3409.

\bibitem[Wal07]{Walther1}
Uli Walther, \emph{Duality and monodromy reducibility of {$A$}-hypergeometric
  systems}, Math. Ann. \textbf{338} (2007), no.~1, 55--74.

\bibitem[Web04]{Weber}
Andrzej Weber, \emph{Weights in the cohomology of toric varieties}, Cent. Eur.
  J. Math. \textbf{2} (2004), no.~3, 478--492 (electronic). \MR{2113544}

\end{thebibliography}
\end{document}